\documentclass[11pt]{amsart}
\usepackage{amssymb}
\usepackage{amscd}
\usepackage{amsmath}
\usepackage[all]{xy}
\usepackage{mathrsfs}
 \usepackage{color}

\theoremstyle{plain}
\newtheorem{theorem}{Theorem}[section]
\newtheorem{proposition}[theorem]{Proposition}
\newtheorem{lemma}[theorem]{Lemma}
\newtheorem{corollary}[theorem]{Corollary}
\newtheorem{conjecture}[theorem]{Conjecture}

\theoremstyle{remark}
\newtheorem{remark}[theorem]{Remark}

\theoremstyle{definition}
\newtheorem{definition}[theorem]{Definition}
\newtheorem{example}[theorem]{Example}
\newtheorem{hypothesis}[theorem]{Hypothesis}

  1

\DeclareMathOperator{\Gal}{Gal}

\DeclareMathOperator{\Hom}{Hom}
\DeclareMathOperator{\End}{End}

\DeclareMathOperator{\N}{N}

\DeclareMathOperator{\im}{im}

\newcommand{\bQ}{\mathbb{Q}}

\newcommand{\bZ}{\mathbb{Z}}
\newcommand{\QQ}{\mathbb{Q}}
\newcommand{\cA}{\mathcal{A}}

\newcommand{\cD}{\mathcal{D}}

\newcommand{\cF}{\mathcal{F}}
\newcommand{\cG}{\mathcal{G}}
\newcommand{\cH}{\mathcal{H}}

\newcommand{\cK}{\mathcal{K}}
\newcommand{\cL}{\mathcal{L}}

\newcommand{\cN}{\mathcal{N}}
\newcommand{\cO}{\mathcal{O}}
\newcommand{\cP}{\mathcal{P}}
\newcommand{\cQ}{\mathcal{Q}}
\newcommand{\cR}{\mathcal{R}}
\newcommand{\cS}{\mathcal{S}}
\newcommand{\cT}{\mathcal{T}}

\newcommand{\fa}{\mathfrak{a}}

\newcommand{\fd}{\mathfrak{d}}
\newcommand{\ff}{\mathfrak{f}}

\newcommand{\fq}{\mathfrak{q}}
\newcommand{\fp}{\mathfrak{p}}
\newcommand{\fr}{\mathfrak{r}}

\newcommand{\fn}{\mathfrak{n}}
\newcommand{\fm}{\mathfrak{m}}
\newcommand{\fQ}{\mathfrak{Q}}

\newcommand{\CC}{\mathbb{C}}

\newcommand{\FF}{\mathbb{F}}
\newcommand{\GG}{\mathbb{G}}

\newcommand{\RR}{\mathbb{R}}

\newcommand{\ZZ}{\mathbb{Z}}
\newcommand{\Z}{\mathbb{Z}}

\newcommand{\id}{\mathrm{id}}

\newcommand{\bz}{\mathbb{Z}}

\newcommand{\ord}{\mathrm{ord}}

\makeatletter
 
  \@addtoreset{equation}{subsection}
\makeatother

\begin{document}

\title[]{On the theory of higher rank\\
 Euler, Kolyvagin and Stark systems, IV:\\
 the multiplicative group}

\author{David Burns, Ryotaro Sakamoto and Takamichi Sano}

\begin{abstract} We describe a refinement of the general theory of higher rank Euler, Kolyvagin and Stark systems in the setting of the multiplicative group over arbitrary number fields. We use the refined theory to prove new results concerning the Galois structure of ideal class groups and the validity of both the equivariant Tamagawa number conjecture and of the `refined class number formula' that has been conjectured by Mazur and Rubin and by Sano. In contrast to previous work in this direction, these results require no hypotheses on the decomposition behaviour of places that are intended to rule out the existence of `trivial zeroes'.
\end{abstract}

\address{King's College London,
Department of Mathematics,
London WC2R 2LS,
U.K.}
\email{david.burns@kcl.ac.uk}

\address{Graduate School of Mathematical Sciences, The University of Tokyo,
3-8-1 Komaba, Meguro-Ku, Tokyo, 153-8914, Japan}
\email{sakamoto@ms.u-tokyo.ac.jp}

\address{Osaka City University,
Department of Mathematics,
3-3-138 Sugimoto\\Sumiyoshi-ku\\Osaka\\558-8585,
Japan}
\email{sano@sci.osaka-cu.ac.jp}

\maketitle

\tableofcontents

\section{Introduction}

\subsection{Discussion of the main results} A general theory of higher rank Euler, Kolyvagin and Stark systems was developed in our articles \cite{bss} and \cite{bss2}, following initial work of Mazur and Rubin in \cite{MRkoly} and subsequent independent work of the first and third authors  in \cite{sbA} and of the second author in \cite{sakamoto}.

In this article we shall now prove, in Theorem \ref{th main}, a refined version of the general theory in the setting of the multiplicative group over arbitrary number fields.

This result shows firstly that, both for a more general class of abelian extensions than is considered in \cite{bss} and \cite{bss2}, and under fewer hypotheses on $p$-adic representations than are used in loc. cit., there exists a canonical `higher Kolyvagin derivative' homomorphism between the modules of higher rank Euler and Kolyvagin systems for the multiplicative group over an arbitrary number field.

The result also shows that, under the same hypotheses, the canonical `regulator' homomorphism that exists between the corresponding modules of Stark and Kolyvagin systems is bijective and also that these modules are each  free of rank one over the relevant Gorenstein algebra.

In terms of the approach taken in \cite{bss} and \cite{bss2}, these are the key facts that one must establish in order to develop a natural (`equivariant') theory of higher rank Euler and Kolyvagin systems.

To demonstrate the effectiveness of the theory that results in this case, we shall then use it to derive three concrete consequences of the assumed validity of the Rubin-Stark Conjecture.

The first two consequences that we describe (in Corollaries \ref{cor main 1} and \ref{cor main 2}) respectively concern the fine Galois structure of ideal class groups, and thereby strongly refined versions of the classical Brumer and Brumer-Stark Conjectures, and the validity of the equivariant Tamagawa number conjecture for the untwisted Tate motive.

These results improve upon the corresponding results that are proved in \cite{bss2} in that the methods developed here allow us to deal with components at $p$-adic characters that are trivial on the Frobenius elements of $p$-adic places.

This improvement is significant since, for example, in the context of CM extensions of totally real fields it allows us to avoid having to effectively assume that the associated Deligne-Ribet $p$-adic $L$-series have no `trivial zeroes'.

In this context we recall that ever since Wiles' seminal work in \cite{wiles}, it has been clear that trivial zeroes play a pivotal role in relation to efforts to deduce Brumer's Conjecture and its variants from the validity of appropriate main conjectures in Iwasawa theory.

In particular, in order to avoid difficult problems related to the order of vanishing of $p$-adic $L$-series, 
 previous investigations of these questions have usually either followed ideas of Greither in \cite{Greither}  and studied special classes of fields in which trivial zeroes can be ruled out or, as in the recent work of Greither and Popescu \cite{GreitherPopescu} and Johnston and Nickel \cite{JN}, studied weaker versions of the conjectures by considering `imprimitive' $L$-series. However, under mild hypotheses, the approach developed here now allows one to avoid all such problems.

The third consequence that we shall derive from Theorem \ref{main} is stated as Corollary \ref{cor main 3} and provides new evidence for the `refined class number formula for $\mathbb{G}_m$' that was independently formulated by Mazur and Rubin in \cite{MRGm} and by the third author in \cite{sano}.

This result strongly improves upon the existing evidence for the conjecture of Mazur, Rubin and Sano since our methods allow us both to consider $p$-components of the conjectural formula at all odd primes $p$ and also to deal with $L$-series of arbitrary order of vanishing.

In contrast, previous results in this direction have either assumed the validity of the relevant case of the equivariant Tamagawa number conjecture, as in \cite[Th. 1.1]{bks1} and the main result of \cite{sano}, or have both avoided considering $p$-components for primes $p$ that divide the degree of the relevant Galois extension and also either assumed that all occurring $L$-series vanish to order one, as in \cite[Th. 10.7]{MRGm}, or assumed the validity of Leopoldt's Conjecture (in addition to that of the Rubin-Stark Conjecture) and only considered a weaker version of the refined class number formula, as in \cite[Th. 11.6]{MRGm}.

Finally, we recall that the Rubin-Stark Conjecture is known to be valid both in the setting of absolutely abelian fields and of abelian extensions of any imaginary quadratic field.

In all such cases, therefore, the results of Corollaries \ref{cor main 1}, \ref{cor main 2} and \ref{cor main 3} are valid unconditionally.

In the case of absolutely abelian fields these results merely give (much) simpler proofs of various existing results, including relevant cases of the main result of Greither and the first author in \cite{BG} and 
several of the main results of Kurihara and the first and third authors in \cite{bks1} (for more details of which see \cite[Rem. 5.2(ii)]{bss2}).

However, in the setting of abelian extensions of imaginary quadratic fields these results are new and can be seen to extend, for example, the main results of Bley in \cite{bley} concerning the equivariant Tamagawa number conjecture and of Gomez in \cite{gomez} concerning refined class number formulas.

The basic contents of this article is as follows. In \S\ref{sec eks} we introduce a natural notion of higher rank `motivic' Euler systems for $\mathbb{G}_m$ and also define associated notions of Kolyvagin and Stark systems in this context. In \S\ref{sec main} we then state the main results of this article and give proofs of Theorem \ref{th main} and Corollaries \ref{cor main 1} and \ref{cor main 2}, with certain auxiliary arguments (that are required  for the proof of Theorem \ref{th main}) being given in an appendix to the article. In \S\ref{sec rcnf} we recall the precise statement of the refined class number formula conjecture (as Conjecture \ref{mrs}) and then, finally, we give a precise statement and proof of our main result (Theorem \ref{main}) concerning this conjecture.

\subsection{Notation}

Throughout this article, $K$ denotes a number field. We fix an algebraic closure $\overline \QQ$ of $\QQ$, and regard $K \subset \overline \QQ$. We set $G_K:=\Gal(\overline \QQ/K)$.

We fix a prime number $p$ and an algebraic closure $\overline \QQ_p$ of $\QQ_p$.

The set of all archimedean places of $K$ is denoted by $S_\infty(K)$ and the set of $p$-adic places by $S_p(K)$. If $E/K$ is an algebraic extension, then the set of places of $K$ that ramify in $E$ is denoted by $S_{\rm ram}(E/K)$.

For each place $v$ of $K$, we fix a place $w$ of $\overline \QQ$ lying above $v$. In this way, for any algebraic extension $E/K$ we fix a place $w$ of $E$ that lies above $v$.

The symbol $\fq$ always means a non-archimedean place (i.e., a prime ideal) of $K$. The fixed place lying above $\fq$ is often denoted by $\fQ$. Sometimes we also use $\fr$ for the notation of a prime ideal of $K$.

For a finite set $S$ of places of $K$ that contains $S_\infty(K)$ and a finite extension $E/K$, we write $\cO_{E,S}$ for the ring of $S_E$-integers of $E$, where $S_E$ is the set of places of $E$ lying above a place in $S$.

For a non-archimedean local field $D$, we denote its residue field by $\FF_D$. If $\fq$ is a prime of $K$, then we denote $\# \FF_{K_\fq}$ by ${\N}\fq$.

For a finite abelian group $G$, we set $\widehat G:=\Hom(G,\CC^\times)$. For each $\chi$ in $\widehat G$, we write $e_\chi$ for the associated idempotent $(\# G)^{-1}\sum_{\sigma \in G}\chi(\sigma)\sigma^{-1}$ of $\CC[G]$. 

For an abelian group $A$ and a non-negative integer $m$, we often abbreviate $A/mA$ to $A/m$. In particular, for a field $E$, we denote $E^\times/(E^\times)^m$ by $E^\times/m$. 

For a commutative ring $R$, an $R$-module $X$ (usually finitely generated) and a non-negative integer $r$, we define
$${\bigcap}_R^r X:=\Hom_R\left({\bigwedge}_R^r \Hom_R(X,R),R \right).$$
We recall the following basic properties of such modules (taken from \cite[App. A]{sbA}).
\begin{itemize}
\item There exists a natural homomorphism
$${\bigwedge}_R^r X \to {\bigcap}_R^r X; \ a \mapsto (\Phi \mapsto \Phi(a))$$
(that is, in general, neither injective nor surjective).
\item Let $\cO$ be a Dedekind domain and $\cQ$ its quotient field. If $R$ is a Gorenstein $\cO$-order in some finite dimensional commutative semisimple $\cQ$-algebra (e.g., $R$ is a finite group ring over $\cO$), we have a natural identification
$${\bigcap}_R^r X \simeq \left\{a \in \cQ \otimes_\cO {\bigwedge}_R^r X \ \middle| \ \Phi(a) \in R \text{ for all }\Phi \in {\bigwedge}_R^r \Hom_R(X,R) \right\}.$$
In this way we can regard ${\bigcap}_R^r X$ as a sublattice of $\cQ \otimes_\cO {\bigwedge}_R^r X.$
\end{itemize}

For a finitely presented $R$-module $X$ and a non-negative integer $j$, the $j$-th Fitting ideal of $X$ is denoted by ${\rm Fitt}_R^j(X)$.

\section{Euler, Kolyvagin and Stark systems}\label{sec eks}

\subsection{The definition of Euler systems for $\GG_m$}\label{sec first}
Let $K$ be a number field. Let $S$ be a finite set of places of $K$ containing $S_\infty(K)$. Let $\mathcal{L}/K$ be a (possibly infinite) abelian extension. We write $\Omega(\mathcal{L}/K)$ for the set of finite extensions of $K$ in $\mathcal{L}$. For each field $E$ in $\Omega(\mathcal{L}/K)$ we set
$$\cG_E:=\Gal(E/K) \text{ and }S(E):=S\cup S_{\rm ram}(E/K).$$
%

\begin{definition}\label{mes def}
Let $r$ be a non-negative integer.
A {\it motivic Euler system} of rank $r$ for $(\mathcal{L}/K,S)$ is a collection
$$u=(u_E)_{E  } \in \prod_{E \in \Omega(\mathcal{L}/K)\setminus \{K\}}  \QQ \otimes_\ZZ {\bigwedge}_{\ZZ[\cG_E]}^r \cO_{E,S(E)}^\times$$
which satisfies the following property (the `norm relation'): for any $E, E' \in \Omega(\mathcal{L}/K)\setminus \{K\}$ with $E\subset E'$, we have
$${\N}_{E'/E}^r(u_{E'})= \left( \prod_{\fq \in S(E')\setminus S(E)} (1-{\rm Fr}_\fq^{-1})\right)u_E \text{ in } \QQ \otimes_\ZZ {\bigwedge}_{\ZZ[\cG_E]}^r \cO_{E, S(E')}^\times.$$
Here ${\N}_{E'/E}^r$ denotes the map $\QQ \otimes_\ZZ {\bigwedge}_{\ZZ[\cG_{E'}]}^r \cO_{E',S(E')}^\times \to \QQ \otimes_\ZZ {\bigwedge}_{\ZZ[\cG_E]}^r \cO_{E,S(E')}^\times$ induced by the norm map ${\N}_{E'/E}: E' \to E$.

\end{definition}

\begin{remark}
For a motivic Euler system $u=(u_E)_E$ to be useful, it should lie in a canonical sublattice of
$\QQ \otimes_\ZZ {\bigwedge}_{\ZZ[\cG_E]}^r \cO_{E,S(E)}^\times$. However, this issue is delicate since $u_E$ does not in general belong to the image of ${\bigwedge}_{\ZZ[\cG_E]}^r \cO_{E,S(E)}^\times $ in $\QQ \otimes_\ZZ {\bigwedge}_{\ZZ[\cG_E]}^r \cO_{E,S(E)}^\times$ (see \cite[\S 4]{rubinstark}). To properly understand such integrality issues, it is usual to fix an auxiliary finite set $T$ of places of $K$ that do not belong to $S$ and are unramified in $\mathcal{L}$ and are such that the group
$$\cO_{E,S(E),T}^\times:=\{a \in \cO_{E,S(E)}^\times \mid {\rm ord}_w(a-1)>0 \text{ for every $w \in T_E$}\}$$
is torsion-free, where $\ord_w$ denotes the normalized additive valuation at $w$. One then considers motivic Euler systems $u$ with the property that $u_E$ belongs to ${\bigcap}_{\ZZ[\cG_E]}^r \cO_{E,S(E),T}^\times$ for all $E$.

In this paper, however, we do not need to explicitly consider `$T$-modifications' as we shall only focus on the `$(p,\chi)$-component' of motivic Euler systems in situations for which the $(p,\chi)$-component of the group $\cO_{E,S(E)}^\times$ is itself torsion-free. For details see Definition~\ref{p integral} and Example~\ref{ex2}(i) below.
\end{remark}


\begin{example}\label{ex1}\

\noindent{}(i) (The cyclotomic Euler system) Suppose $K=\QQ$ and $S=\{\infty\}$. Fix an embedding $\overline \QQ \hookrightarrow \CC$, and regard $\overline \QQ \subset \CC$. Let $\mathcal{L}/\QQ$ be the maximal real abelian extension. Then for any field $E$ in $\Omega(\cL/\QQ)\setminus \{\QQ\}$ we define the cyclotomic unit by
$$\eta_{E}^{\rm cyc}:=\frac 12 \otimes {\N}_{\QQ(\mu_m)/E}(1-\zeta_m) \in \QQ \otimes_\ZZ \cO_{E,S(E)}^\times,$$
where $m=m_E$ is the conductor of $E$ and $\zeta_m:=e^{2\pi i/m}$. The collection $(\eta_{E}^{\rm cyc})_E$ is a motivic Euler system for $(\cL/\QQ,S)$.

\noindent{}(ii) (The elliptic Euler system) Suppose that $K$ is an imaginary quadratic field. Fix a non-zero ideal $\ff$ of $K$ such that $\cO_K^\times \to (\cO_K/\ff)^\times$ is injective and set $S:=\{\infty\}\cup \{\fq \mid \ff\}$. For an ideal $\fm$ of $K$, let $K(\fm)$ denote the ray class field of $K$ modulo $\fm$. As in (i), we regard $\overline \QQ \subset \CC$. Let $\mathcal{L}/K$ be the maximal abelian extension. For any field $E$ in $\Omega(\mathcal{L}/K)$ we define the elliptic unit by
$$\eta_E^{\rm ell}:=(\sigma_\fa- {\N}\fa)^{-1} \cdot \N_{K(\fm\ff)/E}\left( {}_\fa z_{\fm \ff} \right) \in \QQ \otimes_\ZZ \cO_{E,S(E)}^\times,$$
where $\fm$ is the conductor of $E$, $\fa \neq (1)$ is an ideal of $K$ coprime to $6\fm\ff$, $\sigma_\fa \in \cG_E$ is the Artin symbol, ${\N}{\fa}$ is the order of $\cO_K/\fa$, and ${}_\fa z_{\fm\ff} \in \cO_{K(\fm\ff),S(E)}^\times$ is the element defined in \cite[\S 15.5]{katoasterisque} (it coincides with $\psi(1;\fm\ff,\fa)^{-1}$ in \cite{bley}). (Note that $\sigma_\fa - {\N}\fa$ is invertible in $\QQ[\cG_E]$ and so $(\sigma_\fa -{\N}\fa)^{-1} \in \QQ[\cG_E]$ is defined.) One sees that this element is independent of $\fa$. The collection $(\eta_E^{\rm ell})_E$ is a motivic Euler system for $(\mathcal{L}/K,S)$.

\noindent{}(iii) (The Rubin-Stark Euler system) Let $(\mathcal{L}/K,S)$ be any data as in Definition \ref{mes def}. Fix an integer $r$ such that there exists a subset $V:=\{v_1,\ldots,v_r\}$ of $S_\infty(K)$ comprising places that split completely in $\mathcal{L}$ and, in addition, one has $\# S(E) >r$ for every $E$ in $\Omega(\mathcal{L}/K)\setminus \{K\}$. For each such field $E$ one can use the data $(E/K,S(E),V)$ to specify a canonical `Rubin-Stark element' $\eta_{E/K,S(E)}^V$ of $\mathbb{R}\otimes_\ZZ{\bigwedge}_{\ZZ[\cG_E]}^r \cO_{E,S(E)}^\times$. (See \cite[\S 5.1]{bks1} for the definition of such elements but note that we take the set `$T$' in loc. cit. to be empty.) Then the Stark conjecture \cite[Conj. A$'$]{rubinstark} predicts that each $\eta_{E/K,S(E)}^V$ belongs to $\QQ \otimes_\ZZ{\bigwedge}_{\ZZ[\cG_E]}^r \cO_{E,S(E)}^\times$ and if this is valid the collection $(\eta_{E/K,S(E)}^V)_E$ is a motivic Euler system for $(\mathcal{L}/K,S)$. If $K=\QQ$ and $S=\{\infty\}$, then $\eta_{E/\QQ,S(E)}^{\{\infty\}}$ coincides with the cyclotomic unit $\eta^{\rm cyc}_{E}$ in (i) (see \cite[p.79]{tatebook}). If $K$ is an imaginary quadratic field and $S$ is as in (ii), then $\eta_{E/K,S(E)}^{\{\infty\}}$ coincides with $\eta_E^{\rm ell}$ in (ii) (this is verified by Kronecker's limit formula, see \cite[(15.5.1)]{katoasterisque} or \cite[(10)]{bley}).
\end{example}

We now fix an {\it odd} prime number $p$ and a {\it non-trivial} character $\chi: G_K \to \overline \QQ_p^\times$ of finite prime-to-$p$ order, and set
$$L:=\overline \QQ^{\ker \chi}\text{ and }\Delta:=\Gal(L/K).$$
We assume that $S$ contains $S_{\rm ram}(L/K)$ (but do not need to assume that $S$ contains $S_p(K)$). We set
$$\cO:=\ZZ_p[\im \chi].$$
For any $\ZZ[\Delta]$-module $X$, we define its $(p,\chi)$-component by
$$X_\chi:=\cO \otimes_{\ZZ[\Delta]} X,$$
where $\cO$ is regarded as a $\ZZ[\Delta]$-algebra via $\chi$. For an element $a \in X$, we set
$$a^\chi:=1 \otimes a \in X_\chi.$$
Let $\cK/K$ be an abelian pro-$p$-extension. For $F \in \Omega(\cK/K)$, we set
$$U_{F,S(F)}:=(\cO_{LF,S(F)}^\times)_\chi \text{ and }U_{F}:=(\cO_{LF}^\times)_\chi.$$

\begin{definition}\label{p integral}
Let $r$ be a non-negative integer.
A {\it $p$-adic Euler system} of rank $r$ for $(\mathcal{K}/K,S,\chi)$ is a collection
$$c=(c_F)_{F } \in \prod_{F \in \Omega(\cK/K)} {\bigcap}_{\cO[\cG_F]}^r U_{F,S(F)}$$
that has the following property: for all $F$ and $F'$ in $\Omega(\cK/K)$ with $F\subset F'$, we have
$${\N}_{F'/F}^r(c_{F'})= \left( \prod_{\fq \in S(F')\setminus S(F)} (1-{\rm Fr}_\fq^{-1})\right)c_F \text{ in } {\bigcap}_{\cO[\cG_F]}^r U_{F,S(F')}.$$
Such a system $c=(c_F)_F$ is said to be a {\it strict $p$-adic Euler system} if $c_F$ belongs to ${\bigcap}_{\cO[\cG_F]}^r U_{F}$ for every $F$ in $\Omega(\cK/K)$. The set of strict $p$-adic Euler systems of rank $r$ for $(\cK/K,S,\chi)$ will be denoted ${\rm ES}_r(\cK/K,S,\chi)$  and is naturally an $\cO[[\Gal(\cK/K)]]$-module. 
\end{definition}

\begin{remark}\label{twisted rep} Write $\cO(1) \otimes \chi^{-1}$ for the representation of $G_K$ that is equal to $\cO$ as an $\cO$-module and upon which $G_K$ acts via $\chi_{\rm cyc}\chi^{-1}$, where $\chi_{\rm cyc}$ denotes the cyclotomic character of $K$. Then, if $S$ contains $S_p(K)$, Kummer theory induces an identification
$$U_{F,S(F)}=H^1(\cO_{F,S(F)}, \cO(1) \otimes \chi^{-1}).$$
This observation implies that, for any such $S$, the definition of $p$-adic Euler systems of rank $r$ for $(\cK/K,S,\chi)$ coincides with that of Euler systems of rank $r$ for $(\cO(1)\otimes\chi^{-1}, \cK)$ in the sense of \cite[Def. 6.5]{bss} (with $K$ and $S$ fixed).
\end{remark}

\begin{example}\label{ex2}
Set $\mathcal{L}:=L \cK$.

\noindent{}(i) Suppose that a motivic Euler system for $(\mathcal{L}/K,S)$
$$u=(u_E)_{E} \in \prod_{E \in \Omega(\cL/K) \setminus \{ K\}} \QQ \otimes_\ZZ{\bigwedge}_{\ZZ[\cG_E]}^r \cO_{E,S(E)}^\times$$
is given.
For $F \in \Omega(\cK/K)$, set
$$c_F:=u_{LF}^\chi \in \left( \QQ \otimes_\ZZ {\bigwedge}_{\ZZ[\cG_{LF}]}^r \cO_{{LF},S({F})}^\times \right)_\chi= \QQ \otimes_\ZZ {\bigwedge}_{\cO[\cG_F]}^r U_{F,S(F)}.$$
Then, if each $c_F$ belongs to ${\bigcap}_{\cO[\cG_F]}^r U_{F,S(F)}$, the collection
$$c=(c_F)_{F } \in \prod_{F \in \Omega(\cK/K)} {\bigcap}_{\cO[\cG_F]}^r U_{F,S(F)}$$
is a $p$-adic Euler system for $(\cK/K,S,\chi)$.

\noindent{}(ii) Let $(\eta_{E/K,S(E)}^V)_E$ be the Rubin-Stark Euler system for $(\mathcal{L}/K,S)$ discussed in Example~\ref{ex1}(iii). Assume that $U_F$ is $\cO$-free for every $F$ in $\Omega(\cK/K)$. Then the $(p,\chi)$-component of the Rubin-Stark conjecture (see \cite[Conj. B$'$]{rubinstark} or \cite[Conj. 5.1]{bks1}) asserts that
$$\eta_{LF/K,S(F)}^{V,\chi} \in {\bigcap}_{\cO[\cG_F]}^r U_{F,S(F)}.$$
Furthermore, since $V$ is contained in $S_{\infty}(K)$, the argument of \cite[Prop. 6.2(i)]{rubinstark} implies that $\eta_{LF/K,S(F)}^{V,\chi}$ lies in ${\bigcap}_{\cO[\cG_F]}^r U_F$.
Thus, if the Rubin-Stark conjecture is valid for every $F$ in $\Omega(\cK/K)$, then the collection  $(\eta_{LF/K,S(F)}^{V,\chi})_F$ is a strict $p$-adic Euler system.
\end{example}

\subsection{The definitions of Stark and Kolyvagin systems for $\GG_m$}\label{sec koly}

Fix a finite abelian $p$-extension $F/K$ and set
$$\Gamma:=\cG_F=\Gal(F/K).$$
Let $\chi, L, \Delta, \cO$ be as in the previous subsection. For notational simplicity, we set
$$E:=LF.$$
We also fix a finite set $S$ of places of $K$ such that
$$
S_\infty(K) \cup S_{\rm ram}(E/K) \subset S.
$$
Let $m$ be a non-negative integer. Let $H_K$ be the Hilbert class field of $K$ and $K(1)$ the maximal $p$-extension inside $H_K/K$.
We use the following notations:
\begin{itemize}
\item $\cO_m:=\cO/p^m \cO$;
\item $\cP_m:=\{ \fq \notin S\cup S_p(K) \mid \text{$\fq$ splits completely in $E H_K (\mu_{p^m}, (\cO_K^\times)^{1/p^m})$}\}$.
\end{itemize}
For any set $\cQ$ of primes of $K$, we set
$$\cN(\cQ):=\{\text{square-free products of primes in $\cQ$}\}.$$
If $\cQ=\cP_m$, we write
$$\cN_m:=\cN(\cP_m).$$
For $\fq \in \cQ$, let $K(\fq)/K$ be the maximal $p$-extension inside the ray class field modulo $\fq$, and set
$$G_\fq:=\Gal(K(\fq)/K(1)) .$$
%
%
For $\fn \in \cN_m$, we set
$$\nu(\fn):=\# \{\fq \mid \fn\} \text{ (the number of prime divisors of $\fn$)}$$
and
$$
G_\fn:=\bigotimes_{\fq \mid \fn} G_\fq.
$$
Recall that $\Gal(E/K) \simeq \Delta \times \Gamma$.
For any prime $\fq$ of $K$, define
\begin{eqnarray}\label{def v}
v_\fq: E^\times \to \ZZ[\Delta \times \Gamma]; \ a \mapsto \sum_{\sigma \in \Delta \times \Gamma} {\rm ord}_\fQ(\sigma a )\sigma^{-1},
\end{eqnarray}
where $\fQ$ is the fixed place of $E$ lying above $\fq$ and ${\rm ord}_\fQ$ denotes the normalized additive valuation. This map induces
$$(E^\times/p^m)_\chi \to \cO_m[\Gamma],$$
which we denote also by $v_\fq$. We set
$$\cS_m^\fn:=\{a \in (E^\times/p^m)_\chi \mid v_\fq(a)=0 \text{ for every }\fq \nmid \fn\}.$$
If $\fm,\fn \in \cN_m$ and $\fn \mid \fm$, we have an exact sequence
$$0 \to \cS_m^\fn \to \cS_m^\fm \xrightarrow{\bigoplus_{\fq \mid \fm/\fn}v_\fq}\bigoplus_{\fq \mid \fm/\fn} \cO_m[\Gamma].$$
So we obtain a map
$$v_{\fm,\fn}:=\pm {\bigwedge}_{\fq \mid \fm/\fn} v_\fq: {\bigcap}_{\cO_m[\Gamma]}^{r+\nu(\fm)}\cS_m^\fm \to {\bigcap}_{\cO_m[\Gamma]}^{r+\nu(\fn)}\cS_m^\fn$$
for any non-negative integer $r$.
(See \cite[Prop.~A.3]{sbA}.) The sign is appropriately chosen so that $v_{\fm',\fn}=v_{\fm,\fn}\circ v_{\fm',\fm}$ if $\fn \mid \fm \mid \fm'$. (See \cite[\S 3.1]{sbA} for the precise choice.)

\begin{definition}
Let $\cQ $ be a subset of $\cP_m$. Then the $\cO_m[\Gamma]$-module of Stark systems of rank $r$ for $(F/K,\chi,m,\cQ)$ is defined by the inverse limit
$${\rm SS}_r(\cQ)_m={\rm SS}_r(F/K,\chi,\cQ)_m:=\varprojlim_{\fn \in \cN(\cQ)} {\bigcap}_{\cO_m[\Gamma]}^{r+\nu(\fn)}\cS^\fn_m$$
with transition maps $v_{\fm,\fn}$.

\end{definition}

For a Stark system $\epsilon=(\epsilon_\fn)_\fn$ in ${\rm SS}_r(F/K,\chi,\cQ)_m$ and a non-negative integer $j$, we define
$$
I_j(\epsilon):=\sum_{\fn \in \cN(\cQ), \ \nu(\fn)=j} \im(\epsilon_\fn) \subset \cO_m[\Gamma].
$$

For $\fq \in \cP_m$, let $l_\fq: E \to E_\fQ = K_\fq$ be the localization map. We define
\begin{eqnarray}\label{def l}
\widetilde l_\fq: E^\times \to \ZZ[\Delta \times \Gamma] \otimes_\ZZ K_\fq^\times; \ a \mapsto \sum_{\sigma \in \Delta \times \Gamma}\sigma^{-1}\otimes l_\fq(\sigma a).
\end{eqnarray}
This map induces
$$\widetilde l_\fq: (E^\times/p^m)_\chi \to \cO_m[\Gamma]\otimes_\ZZ K_\fq^\times.$$
Let ${\rm rec}_\fq: K_\fq^\times \to \Gal(K(\fq)_\fQ/K_\fq)\simeq G_\fq$ denote the local reciprocity map at $\fq$.
We define
\begin{eqnarray}\label{def phi}
\varphi_\fq : E^\times \to \ZZ[\Delta\times \Gamma]\otimes_\ZZ G_\fq
\end{eqnarray}
to be the composition $(\id \otimes {\rm rec}_\fq) \circ \widetilde l_\fq$.
This map induces
$$\varphi_\fq: (E^\times/p^m)_\chi \to \cO_m[\Gamma]\otimes_\ZZ G_\fq.$$
Let $\Pi_\fq \subset K_\fq^\times/p^m$ be the subgroup defined by
$$\Pi_\fq:=\ker(K_\fq^\times/p^m \to K(\fq)_\fQ^\times/p^m ).$$
Then we see that the valuation map $K_\fq^\times \to \ZZ$ induces an isomorphism
$$\Pi_\fq \xrightarrow{\sim} \ZZ/p^m$$
and we have a canonical decomposition
$$K_\fq^\times/p^m = \Pi_\fq \times \FF_{K_\fq}^\times/p^m.$$
For $\fn \in \cN_m$, we define
$$\cS(\fn)_m:=\{a \in \cS^\fn_m \mid \widetilde l_\fq(a) \in \cO_m[\Gamma]\otimes_\ZZ \Pi_\fq \text{ for every }\fq \mid \fn\}.$$
For $\fq \mid \fn$, we define
$$\cS_\fq(\fn)_m:=\cS(\fn)_m \cap \cS^{\fn/\fq}_m =\ker(\cS(\fn)_m \xrightarrow{\widetilde l_\fq} \cO_m[\Gamma]\otimes_\ZZ K_\fq^\times).$$
Note that, for any $\fn \in \cN_m$ and $\fq \mid \fn$, the maps $v_\fq$ and $\varphi_\fq$ induce maps
$$v_\fq: {\bigcap}_{\cO_m[\Gamma]}^r \cS(\fn)_m \otimes_\ZZ G_\fn \to {\bigcap}_{\cO_m[\Gamma]}^{r-1} \cS_\fq(\fn)_m \otimes_\ZZ G_\fn$$
and
$$\varphi_\fq: {\bigcap}_{\cO_m[\Gamma]}^r \cS(\fn/\fq)_m \otimes_\ZZ G_{\fn/\fq} \to {\bigcap}_{\cO_m[\Gamma]}^{r-1} \cS_\fq(\fn)_m \otimes_\ZZ G_\fn$$
respectively, for any positive integer $r$.

\begin{definition}\label{def koly}
Let $r$ be a positive integer and $\cQ$ a subset of $\cP_m$.
A Kolyvagin system of rank $r$ for $(F/K,\chi,m,\cQ)$ is a collection
$$\kappa=(\kappa_\fn)_\fn \in \prod_{\fn \in \cN(\cQ)} {\bigcap}_{\cO_m[\Gamma]}^r \cS(\fn)_m \otimes_\ZZ G_\fn$$
that has the following property: for every $\fn \in \cN(\cQ)$ and $\fq \mid \fn$, one has
$$v_\fq(\kappa_\fn)=\varphi_\fq(\kappa_{\fn/\fq}) \text{ in } {\bigcap}_{\cO_m[\Gamma]}^{r-1} \cS_\fq(\fn)_m \otimes_\ZZ G_\fn.$$
The set of all Kolyvagin systems of rank $r$ for $(F/K,\chi,m,\cQ)$ is denoted by ${\rm KS}_r(\cQ)_m = {\rm KS}_r(F/K,\chi,\cQ)_m$ and is naturally an $\cO_m[\Gamma]$-module.
\end{definition}

We fix a positive integer $r$ and a subset $\cQ \subset \cP_m$.
We construct a `regulator map'
$$\cR_{r,m}: {\rm SS}_r(\cQ)_m \to {\rm KS}_r(\cQ)_m$$
as follows. For each $\epsilon=(\epsilon_\fn)_\fn$ in ${\rm SS}_r(\cQ)_m = \varprojlim_{\fn \in \cN(\cQ)}{\bigcap}_{\cO_m[\Gamma]}^{r+\nu(\fn)}\cS_m^\fn$, we set
$$\cR_{r,m}(\epsilon)_\fn:=\pm \left( {\bigwedge}_{\fq \mid \fn}\varphi_\fq\right)(\epsilon_\fn) \in {\bigcap}_{\cO_m[\Gamma]}^r \cS(\fn)_m \otimes_\ZZ G_\fn,$$
where the sign is again specified as in \cite[\S 4.2]{sbA}.

Then one checks that $\cR_{r,m}(\epsilon):=(\cR_{r,m}(\epsilon_\fn))_\fn$ belongs to
${\rm KS}_r(\cQ)_m$ (see \cite[Prop.~4.3]{sbA}) and the assignment $\epsilon\mapsto \cR_{r,m}(\epsilon)$ defines the homomorphism $\cR_{r,m}$.

\subsection{Limits of Stark and Kolyvagin systems}\label{limits section}

We write $\omega$ for the Teichm\"uller character of $K$.

In the sequel we shall assume the following hypothesis.

\begin{hypothesis}\label{hyp chi}\
\begin{itemize}
\item[(i)] $\chi\notin \{1, \omega\}$;
\item[(ii)] $\chi^2 \neq \omega$ if $p=3$;
\item[(iii)] $\chi(\fq) \neq 1$ (i.e., $\fq$ does not split completely in $L$) for every
$\fq \in S_{\rm ram}(F/K)$;
\item[(iv)] $r=\# \{ v \in S_\infty(K) \mid \text{$v$ splits completely in $L$}\}>0$.
\end{itemize}
\end{hypothesis}

\begin{remark}\label{rem hyp}
One can check that Hypothesis \ref{hyp chi}(i) implies that the $\cO$-module $U_{F'}$ is free for any finite abelian $p$-extension $F'$ of $K$. 
Also, by Lemma~\ref{lem:ur} below,
Hypothesis~\ref{hyp chi} implies each of the hypotheses (H$_0$), (H$_1$), (H$_2$), (H$_3$), (H$_4$) and (H$_5^{\rm u}$) that are listed in \cite[\S 3.1.3]{bss2} for the representation $T=\cO(1)\otimes \chi^{-1}$ and the field $F$
(see \cite[Lem.~5.3 and 5.4]{bss2}).

\end{remark}

\begin{lemma}\label{lem:ur}
Let $T := \cO(1) \otimes \chi^{-1}$.
Assume Hypotheses~\ref{hyp chi}(i), (ii) and (iii).
Then Hypothesis~(H$_5^{\rm u}$) in \cite[\S 3.1.3]{bss2} is satisfied by the pair $(T,F)$.
\end{lemma}
\begin{proof}
Let $\Bbbk$ denote the residue field of $\cO$ and $\overline{T} := T \otimes_{\cO} \Bbbk$ the residual representation of $T$.
Put $(-)^{\vee} := \Hom(-, \bQ_{p}/\bZ_{p})$.

First, we note that Hypotheses~\ref{hyp chi}(i) and (ii) imply Hypothesis~(H$_1$), ~(H$_2$) and (H$_{3}$)
(see \cite[Lem.~5.3 and 5.4]{bss2}) and that, for any prime $\fq$ of $K$, the module
$H^{0}(K_{\fq}, \overline{T}^{\vee}(1))$ vanishes if one has $\chi(\fq) \neq 1$.

Hence, by \cite[Lem.~3.10]{bss2}, we only need to show that,
for any prime $\fp \in S_{p}(K)$ with $\chi(\fp) = 1$,
the map
\[
H^{1}(K_{\fp},\overline{T})/H^{1}_{\cF_{\rm ur}}(K_{\fp},\overline{T}) \to H^{1}(K_{\fp},\overline{T} \otimes_{\bZ_{p}} \bZ_{p}[\Gamma])/H^{1}_{\cF_{\rm ur}}(K_{\fp},\overline{T} \otimes_{\bZ_{p}} \bZ_{p}[\Gamma])
\]
induced by $\Bbbk \hookrightarrow \Bbbk[\Gamma]$ is injective.
Here $\cF_{\rm ur}$ denotes the unramified Selmer structure defined in \cite[Def.~5.1]{MRselmer}, namely,
for $M \in \{T, T \otimes_{\bZ_{p}} \bZ_{p}[\Gamma] \}$, the group
$H^{1}_{\cF_{\rm ur}}(K_{\fp}, M \otimes_{\cO} \Bbbk)$ is defined to be the image of
the universal norm subgroup
\[
H^{1}_{\cF_{\rm ur}}(K_{\fp}, M) := \bigcap_{J}{\rm Cor}_{J/K_{\fp}}(H^{1}(J, M))
\]
where $J$ runs over all finite unramified extensions of $K_{\fp}$.
Since $\chi(\fp) = 1$, Kummer theory implies that
\begin{align*}
H^{1}(K_{\fp}, T) &= K_{\fp}^{\times, \wedge} \otimes_{\bZ_{p}} \cO, \quad
H^{1}(K_{\fp}, T \otimes_{\bZ_{p}} \bZ_{p}[\Gamma]) = \bigoplus_{\mathfrak{P} \mid \fp}F_{\mathfrak{P}}^{\times, \wedge} \otimes_{\bZ_{p}} \cO,
\\
H^{1}_{\cF_{\rm ur}}(K_{\fp}, T) &= \cO_{K_{\fp}}^{\times, \wedge} \otimes_{\bZ_{p}} \cO, \quad
H^{1}_{\cF_{\rm ur}}(K_{\fp}, T \otimes_{\bZ_{p}} \bZ_{p}[\Gamma]) = \bigoplus_{\mathfrak{P} \mid \fp}\cO_{F_\mathfrak{P}}^{\times, \wedge} \otimes_{\bZ_{p}} \cO,
\end{align*}
where $(-)^{\wedge}$ denotes pro-$p$-completion and $\mathfrak{P}$ runs over all primes of $F$ that divide $\fp$.
Hence we conclude that
\begin{align*}
H^{1}(K_{\fp},\overline{T})/H^{1}_{\cF_{\rm ur}}(K_{\fp},\overline{T}) &= (K_{\fp}^{\times, \wedge}/\cO_{K_{\fp}}^{\times, \wedge}) \otimes_{\bZ_{p}} \Bbbk,
\\
H^{1}(K_{\fp},\overline{T} \otimes_{\bZ_{p}} \bZ_{p}[\Gamma])/H^{1}_{\cF_{\rm ur}}(K_{\fp},\overline{T} \otimes_{\bZ_{p}} \bZ_{p}[\Gamma])
 &= \bigoplus_{\mathfrak{P}\mid \fp}(F_{\mathfrak{P}}^{\times, \wedge}/\cO_{F_{\mathfrak{P}}}^{\times, \wedge}) \otimes_{\bZ_{p}} \Bbbk.
\end{align*}
By Hypothesis~\ref{hyp chi}(iii) and $\chi(\fp) = 1$, $F/K$ is unramified at $\fp$ and hence the natural diagonal map
$$
(K_{\fp}^{\times, \wedge}/\cO_{K_{\fp}}^{\times, \wedge}) \otimes_{\bZ_{p}} \Bbbk
\to
\bigoplus_{\mathfrak{P} \mid \fp}(F_{\mathfrak{P}}^{\times, \wedge}/\cO_{F_\mathfrak{P}}^{\times, \wedge}) \otimes_{\bZ_{p}} \Bbbk
$$
is injective, which completes the proof.
\end{proof}

Under Hypothesis~\ref{hyp chi}, it is shown in \cite[\S 4.3]{bss} that for any non-negative integer $m$ there exists a natural isomorphism
${\rm SS}_r(\cP_m)_m \xrightarrow{\sim} {\rm SS}_r(\cP_{m+1})_m$ and a natural surjection ${\rm SS}_r(\cP_{m+1})_{m+1} \twoheadrightarrow {\rm SS}_r(\cP_{m+1})_m$
and hence also a natural surjective homomorphism
$$\pi_m: {\rm SS}_r(\cP_{m+1})_{m+1} \twoheadrightarrow {\rm SS}_r(\cP_{m})_m.$$

In a similar way, it is shown in \cite[\S 5.5]{bss} that Hypothesis \ref{hyp chi} implies the existence of canonical homomorphisms
$$
\pi_m': {\rm KS}_{r}(\cP_{m+1})_{m+1} \to {\rm KS}_{r}(\cP_m)_m
$$
that lie in commutative diagrams of the form
\begin{equation}\label{transition morphisms2}
\xymatrix{
{\rm SS}_r(\cP_{m+1})_{m+1}  \ar[d]_{\pi_m} \ar[r]^{\cR_{r,m+1}} & {\rm KS}_{r}(\cP_{m+1})_{m+1} \ar[d]^{\pi'_m}
\\
{\rm SS}_r(\cP_{m})_{m}  \ar[r]^{\cR_{r,m}} & {\rm KS}_{r}(\cP_m)_m.
}
\end{equation}
Thus, writing
$${\rm SS}_r(F/K,\chi):=\varprojlim_m {\rm SS}_r(\cP_m)_m \,\,\text{ and }\,\, {\rm KS}_r(F/K,\chi):=\varprojlim_m {\rm KS}_r(\cP_m)_m $$
for the inverse limits with respect to the respective transition morphisms $\pi_m$ and $\pi'_m$, one finds that the regulator maps $(\cR_{r,m})_m$ induce a canonical homomorphism of $\cO[\Gamma]$-modules
$$\cR_r: {\rm SS}_r(F/K,\chi) \to {\rm KS}_r(F/K,\chi).$$

For a Stark system $\epsilon=(\epsilon_m)_m$ in ${\rm SS}_r(F/K,\chi)$ and a non-negative integer $j$, we define an ideal $I_j(\epsilon)$ of $\cO[\Gamma]$ by setting
$$I_j(\epsilon):=\varprojlim_m I_j(\epsilon_m) \subset \varprojlim_m \cO_m[\Gamma]=\cO[\Gamma].$$

\begin{remark}\label{rem unr}
Set $\cT:={\rm Ind}_{G_K}^{G_F}(\cO(1)\otimes \chi^{-1})$, where $\cO(1)\otimes \chi^{-1}$ is the representation of $G_K$ discussed in Remark \ref{twisted rep}. Let $\cF_{\rm ur}$ denote the unramified Selmer structure on $\cT$, as defined in \cite[Def.~5.1]{MRselmer} (or \cite[Exam.~2.3]{bss2}). Then, the natural Kummer theory isomorphism
$$
( E^\times )_\chi \simeq H^1(K,\cT)
$$
combines with the argument of Mazur and Rubin in \cite[\S 5.2]{MRselmer} to imply that the modules ${\rm SS}_r(F/K,\chi)$ and ${\rm KS}_r(F/K,\chi)$ defined above respectively coincide with the modules ${\rm SS}_r(\cT,\cF_{\rm ur})$ and ${\rm KS}_r(\cT,\cF_{\rm ur})$ that are defined in \cite[Def. 4.11 and 5.25]{bss}.
\end{remark}

\section{Statements of the main results}\label{sec main}

In this section we fix data $p,\chi,L,\Delta, \cO, \cK$ as in \S \ref{sec first} and data $F, \Gamma, S$ as in \S \ref{sec koly}.
We set $E := FL$ and assume that $S$ contains $S_{\rm ram}(E/K)$. We also assume that the field $\cK$ contains $F$ and the maximal $p$-extension inside the ray class field modulo $\fq$ for all but finitely many primes $\fq $ of $K$. Finally, we recall that $r$ denotes $\# \{ v \in S_\infty(K) \mid \text{$v$ splits completely in $L$}\}$.

We consider the following hypothesis.
\begin{hypothesis}\label{hyp}
Either
\begin{itemize}
\item[(i)] $\cK$ contains a non-trivial $\ZZ_p$-power extension $K_\infty$ of $K$ in which no finite places split completely and $S_{\rm ram}(K_\infty/K)\subset S$, or
\item[(ii)] $L \not \subset K(\mu_p)$.
\end{itemize}
\end{hypothesis}

\begin{remark}\label{ntz}  Hypothesis \ref{hyp}(i) is a standard assumption in the theory of Euler systems for $p$-adic representations. Unfortunately, however, in the setting of Euler systems for $\GG_m$ it is a rather strong restriction since it forces $S$ to contain any $p$-adic place $\mathfrak{p}$ that ramifies in $K_\infty/K$ and then Hypothesis \ref{hyp chi}(iii) requires that $\chi(\mathfrak{p})$ is non-trivial for all such $\mathfrak{p}$. In particular, in the case that $K$ is totally real and $K_\infty$ contains the cyclotomic $\Z_p$-extension of $K$ one must assume that there are no `trivial zeros' of any associated $p$-adic $L$-functions.

Hypothesis \ref{hyp}(ii) is an adequate substitute for (i) that avoids this issue, as first observed by Rubin in \cite[\S 9.1]{R}.
In addition, since $L \not \subset K(\mu_p)$ if and only if $\chi$ is not a power of $\omega$, it is clear that Hypothesis~\ref{hyp}(ii) is a very mild assumption.
\end{remark}

\subsection{The main results} We can now state the main results of this article.

\begin{theorem}\label{th main}
Assume that $\chi$, $S$, $F$ and $r$ satisfy Hypothesis \ref{hyp chi}.
\begin{itemize}
\item[(i)] If Hypothesis \ref{hyp} is valid, then there exists a canonical homomorphism of $\mathcal{O}[[G_K]]$-modules
$$\cD_{F,r}: {\rm ES}_r(\cK/K,S,\chi) \to {\rm KS}_r(F/K,\chi).$$

\item[(ii)] In all cases, the regulator map induces a canonical isomorphism of $\cO[\Gamma]$-modules
$$\cR_r: {\rm SS}_r(F/K,\chi)\xrightarrow{\sim} {\rm KS}_r(F/K,\chi).$$
Furthermore, these modules are both free $\cO[\Gamma]$-modules of rank one
and we have
\[
\im (\kappa_1) = {\rm Fitt}_{\cO[\Gamma]}^0({\rm Cl}(E)_{\chi}) = {\rm Fitt}_{\ZZ[G]}^0({\rm Cl}(E))_\chi.
\]
for any basis $\kappa$ of ${\rm KS}_r(F/K,\chi)$. Here ${\rm Cl}(E)$ denotes the ideal class group of $E$ and $G :=  \Gal(E/K)=\Delta \times \Gamma$.
\end{itemize}
\end{theorem}

\begin{remark} We shall refer to the map $\cD_{F,r}$ in Theorem \ref{th main}(i) as the `$F$-relative $r$-th order Kolyvagin derivative' homomorphism for $\mathbb{G}_m$ and $\cK/K$. Its construction will be given explicitly in \S \ref{koly der} below (see, in particular, Proposition \ref{unram ks}).
\end{remark}

\begin{remark} In the setting of Euler systems for general $p$-adic representations (rather than for $\mathbb{G}_m$) an algebraic construction of Euler systems given by the first and third author in \cite{sbA} allows one to show that the higher Kolyvagin derivative homomorphism constructed in loc. cit. is surjective. By adapting the construction of Euler systems in \cite{sbA} to the case at hand one can similarly prove that the homomorphism $\cD_{F,r}$ in Theorem \ref{th main}(i) is surjective. The details of this argument will be given in \cite{bdss}. \end{remark}



{\it In the rest of this section, in order to describe some consequences of Theorem~\ref{th main}, we shall assume the validity of the Rubin-Stark conjecture for all abelian extensions of $K$.}

Let $V := \{ v \in S_\infty(K) \mid \text{$v$ splits completely in $L$}\}$.
Then, for any non-trivial finite abelian extension $F'/K$ with $\#S(F') > \# V =:r$, we have
the Rubin-Stark element $\eta_{LF'/K,S(F')}^{V,\chi}$. In particular, if $\#S > r$,
we then write
\[ \eta^{\rm RS}:=(\eta_{LF'/K,S(F')}^{V,\chi})_{F'}\]
for the Rubin-Stark Euler system in
${\rm ES}_r(\cK_p/K,S,\chi)$, as described in Example~\ref{ex2}(ii).
Here we take $\cK_p$ to be sufficiently large so that it contains the maximal abelian pro-$p$ extension of $K$ unramified outside $p$.
We note, in particular, that $\eta_F^{\rm RS}=\eta_{E/K,S}^{V,\chi}$.

\begin{corollary}\label{cor main 1}
Let $S = S_{\infty}(K) \cup S_{\rm ram}(E/K)$.
Assume that
\begin{itemize}
\item $\#S > r$ and
\item either \begin{itemize}
\item there exists a non-trivial $\ZZ_p$-power extension $K_\infty$ of $K$ in which no finite places split completely and $S_{\rm ram}(K_\infty/K) \subset S_{\rm ram}(E/K)$, or
\item $L \not\subset K(\mu_{p})$.
\end{itemize}
\end{itemize}
Then, if Hypothesis~\ref{hyp chi} is also satisfied, the following claims are valid.
\begin{itemize}
\item[(i)] For every non-negative integer $j$ one has
$$I_j(\cR_r^{-1}(\cD_{F,r}(\eta^{\rm RS})))={\rm Fitt}_{\ZZ[G]}^j({\rm Cl}(E))_\chi.$$
In particular, one has
\begin{eqnarray*}\label{image fitt}
\im(\eta_{F}^{\rm RS} )={\rm Fitt}_{\ZZ[G]}^0({\rm Cl}(E))_\chi.
\end{eqnarray*}
\item[(ii)] There exists a canonical isomorphism of finite $\mathcal{O}[\Gamma]$-modules 
\[ \frac{{\bigcap}_{\cO[\Gamma]}^r U_F}{\cO[\Gamma]\cdot\eta_{F}^{\rm RS}} \simeq \Hom_{\ZZ}\left(\frac{\ZZ[G]}
{{\rm Fitt}_{\ZZ[G]}^0({\rm Cl}(E))},\frac{\QQ}{\ZZ}\right)_{\!\!\chi}.\]
\end{itemize}
\end{corollary}

\begin{remark}\
\begin{itemize}
\item[(i)] The assumption ensures that Hypothesis~\ref{hyp} is satisfied for $\cK_{p}$.
\item[(ii)] If $\# (S_{\infty}(K) \cup S_{\rm ram}(E/K)) = r$, then $S_{\rm ram}(E/K)$ is empty and so both $F$ and $L$ are contained in the Hilbert class field $H_{K}$ of $K$.
Hence, for all but finitely many characters $\chi \colon G_{K} \to \overline{\bQ}_p^{\times}$ of finite order,
the condition $\#(S_{\infty}(K) \cup S_{\rm ram}(E/K)) > r$ is satisfied.
\end{itemize}
\end{remark}

In the next result, for any $\cO$-order $\cA$ in $\QQ[G]_\chi$ we write ${\rm TNC}(h^0({\rm Spec}(E)),\cA)$ for the equivariant Tamagawa number conjecture for the pair $(h^0({\rm Spec}(E)),\cA)$, where the motive $h^0({\rm Spec}(E))$ is regarded as defined over $K$. (A precise statement of this conjecture will be given in \S\ref{second proofs} below.)

We also write $\cR_E$ for the order in $\QQ[G]$ given by
$$\{ x \in \QQ[G] \mid x \cdot {\rm Fitt}_{\ZZ[G]}^0({\rm Cl}(E)) \subset {\rm Fitt}_{\ZZ[G]}^0({\rm Cl}(E))\}.$$
%

\begin{corollary}\label{cor main 2}
 Let $S = S_{\infty}(K) \cup S_{\rm ram}(E/K)$.
Assume that
\begin{itemize}
\item $\#S > r$,
\item Hypotheses \ref{hyp chi}, and
\item either \begin{itemize}
\item there exists a non-trivial $\ZZ_p$-power extension $K_\infty$ of $K$ in which no finite places split completely and $S_{\rm ram}(K_\infty/K) \subset S_{\rm ram}(E/K)$, or
\item $L \not\subset K(\mu_{p})$.
\end{itemize}
\end{itemize}
If ${\rm Cl}(L)_\chi$ vanishes, then the conjecture ${\rm TNC}(h^0({\rm Spec}(E)),\ZZ[G]_\chi)$ is valid. In all cases, ${\rm TNC}(h^0({\rm Spec}(E)),\cR_{E,\chi})$ is valid.
\end{corollary}

\begin{remark} Corollaries \ref{cor main 1} and \ref{cor main 2} are a significant refinement of the corresponding results in \cite[Th.~5.1]{bss2} since the homomorphism $\chi$ is here allowed to be trivial on the decomposition subgroup of any $p$-adic place of $K$ that is unramified in $F$. 
\end{remark}

\begin{remark} The Rubin-Stark Conjecture is known to be valid in the context of abelian extensions of either $\QQ$ or of imaginary quadratic fields (cf. Example \ref{ex1}(iii)) and so Corollaries \ref{cor main 1} and \ref{cor main 2} are unconditionally valid in both of these cases. In particular, in the setting of abelian extensions of an imaginary quadratic field $K$, Corollary \ref{cor main 2} improves upon the main result of Bley in \cite{bley} since it avoids the assumption that $p$ splits in $K/\QQ$ which plays a key role in the argument of loc. cit.\end{remark}

\begin{remark} Aside from the special case that ${\rm Cl}(L)_\chi$ vanishes, Corollary \ref{cor main 2} leaves one with the problem of obtaining an explicit description of the order $\cR_{E,\chi}$ and this is in general difficult. Nevertheless, one can make some concrete remarks.

\noindent{}(i) The argument of \cite[Lem. 4.22]{bss2} implies $\cR_{E,\chi}$ is equal to $\ZZ[G]_\chi$ if and only if ${\rm Cl}(E)_\chi$ is a cohomologically-trivial $G$-module. It can be shown that Hypothesis \ref{hyp chi}(iii) implies that this condition is satisfied if and only if for every subgroup $P$ of $G$ of order $p$, the natural `inflation' map ${\rm Cl}(E^P)_\chi\to {\rm Cl}(E)_\chi$ is injective.

\noindent{}(ii) In all cases one has $\cR_{E,\chi}\subset \mathcal{M}_{E,\chi}$, with $\mathcal{M}_E$ the integral closure of $\ZZ$ in $\QQ[G]$. This fact combines with Corollary \ref{cor main 2} and the observation made in \cite[Rem. 5.2(iv)]{bss2} to show that, under the hypotheses that are fixed in this section, the Rubin-Stark Conjecture directly implies the $p$-part of the Strong-Stark Conjecture of Chinburg \cite[Conj.~2.2]{chin83} for all characters of $G$ that extend $\chi$.
\end{remark}

The next result refers to the `refined class number formula' that was independently conjectured by Mazur and Rubin in \cite[Conj. 5.2]{MRGm} and the third author in \cite[Conj. 3]{sano}. A precise statement of this result will be given in Theorem \ref{main} below.

\begin{corollary}[{Theorem \ref{main}}]\label{cor main 3} If $F $ contains $K(1)$, then for every ideal $\fn$ in $\cN_0$ the refined class number formula conjecture of Mazur-Rubin and Sano is valid for the data $(F/K,S ,\chi,\fn)$.
\end{corollary}

\begin{remark} The refined class number formula is stated precisely as Conjecture \ref{mrs} below and was originally formulated as a generalization of a conjecture of Darmon from \cite{D} (for more details see Remark \ref{rem4} and the beginning of \S\ref{idea}). If $r=1$ and $F=K$, then Mazur and Rubin have given a full proof of this conjecture in certain cases (see \cite[\S 10]{MRGm}) but if $r > 1$ they are only able to give partial evidence for the conjecture and, even then, only under the assumption that $F = K$ and Leopoldt's Conjecture is valid (see  \cite[\S 11]{MRGm}). In contrast, by using the methods developed in \cite{sbA}, \cite{bss} and \cite{bss2} we are now able both to deal with the case $r > 1$ without assuming Leopoldt's Conjecture and also to treat the general (and much more difficult) case that $F$ is not equal to $K$. \end{remark}

\subsection{The proof of Theorem \ref{th main}}

%
To construct the homomorphism $\cD_{F,r}$, it is sufficient to construct, for each natural number $m$, a canonical homomorphism
$${\rm ES}_r(\cK/K,S,\chi) \to {\rm KS}_r(\cP_m)_m$$
that is compatible with the transition morphisms $\pi_m'$ in diagram (\ref{transition morphisms2}) as $m$ varies.

In addition, by the general reduction arguments that are used in \cite[\S6.5]{bss} to prove
\cite[Th.~6.12]{bss}, the functorial behaviour of exterior power biduals allows one to deduce the construction of such homomorphisms for general $r$ from the special case that $r=1$.

In the case $r=1$, we shall give an explicit construction of suitable homomorphisms in the Appendix and, in this way, we obtain a proof of claim (i).


We note that, for any $\kappa \in {\rm KS}_{r}(F/K,\chi)$, the ideal $I_0(\cR_r^{-1}(\kappa))$ is equal to the image of $\kappa_{1}$. Hence claim (ii) follows directly from Remarks~\ref{rem hyp} and \ref{rem unr} and the argument of \cite[Th.~3.6(ii) and (iii)]{bss2} (see also \cite[Prop.~5.5]{MRselmer} or \cite[Ex.~2.7(i)]{bss2}).
\qed

\subsection{The proof of Corollaries \ref{cor main 1} and \ref{cor main 2}}

\subsubsection{}We consider the representation $\cT := {\rm Ind}_{G_K}^{G_F} (\cO(1) \otimes \chi^{-1})$ and Selmer structure $\cF_{\rm ur}$ discussed in Remark \ref{rem unr}.

From that remark we know that ${\rm KS}_r(F/K,\chi)$ coincides with the module ${\rm KS}_r(\cT,\cF_{\rm ur})$ defined in \cite{bss} and so
 Theorem \ref{th main}(i) implies that $\cD_{F,r}(c)$ belongs to ${\rm KS}_r(\cT,\cF_{\rm ur})$ for every $c$ in ${\rm ES}_r(\cK_p/K,S,\chi)$.

Given this containment, and the fact that $r > 0$ (by Hypothesis \ref{hyp chi}(iv)), the first assertion of claim (i) of Corollary \ref{cor main 1} follows directly from the argument in the proof of \cite[Th. 5.1(i)]{bss2}.

The second assertion of claim (i) is then an immediate consequence of the first assertion and the fact that $I_0(\cR_r^{-1}(\cD_{F,r}(\eta^{\rm RS})))$ is equal to the image of $\eta_{F}^{\rm RS}$.

%
%

To complete the proof of Corollary \ref{cor main 1} it suffices to note that claim (ii) follows directly upon combining the equality $\im(\eta_{F}^{\rm RS} )={\rm Fitt}_{\ZZ[G]}^0({\rm Cl}(E))_\chi$ with the argument used by Tsoi and the first and third authors to prove \cite[Th. 3.27(i)]{bst}.
\qed


\subsubsection{}\label{second proofs} Before proving Corollary \ref{cor main 2}, we give a precise statement of ${\rm TNC}(h^0({\rm Spec}(E)),\cA)$ for any $\cO$-order $\cA$ in $\QQ[G]_\chi$.

To do this we recall that $S = S_{\infty}(K) \cup S_{\rm ram}(E/K)$ and that the canonical `Weil-\'etale cohomology' complex
\[ C_{E,S} := R\Hom_\ZZ(R\Gamma_{c}((\mathcal{O}_{E,S})_{\mathcal{W}}, \ZZ),\ZZ)[-2]\]
that is defined by Kurihara and the first and third authors in \cite[\S2.2]{bks1} is a perfect complex of $\ZZ[G]$-modules that is acyclic outside degrees zero and one and such that $H^0(C_{E,S})= \mathcal{O}_{E,S}^\times$ and $H^1(C_{E,S})$ lies in a canonical exact sequence
\begin{equation}\label{can ses2} 0 \to {\rm Cl}_{S}(E) \to H^1(C_{E,S}) \to X_{E,S} \to 0.\end{equation}
Here ${\rm Cl}_{S}(E)$ is the ideal class group of $\mathcal{O}_{E,S}$ and for any set of places $\Sigma$ of $K$ we write $X_{E,\Sigma}$ for the submodule of the free abelian group $Y_{E,\Sigma}$ on the set of places $\Sigma_E$ comprising elements whose coefficients sum to zero.

In particular, the Dirichlet regulator map gives a canonical isomorphism
\[ \lambda_{E}: \RR\otimes_\ZZ H^0(C_{E,S}) = \RR\otimes_\ZZ \mathcal{O}_{E,S}^\times \to \RR\otimes_\ZZ X_{E,S} = \RR\otimes_\ZZ H^1(C_{E,S})\]
and hence (via the approach of \cite[\S3.2]{bks1}) induces a canonical isomorphism of graded $\mathbb{R}[G]$-modules $\vartheta_{\lambda_{E}} : \RR\otimes_\ZZ{\det}_{\ZZ[G]}( C_{E,S}) \to (\RR[G],0),$
where ${\det}_{\ZZ[G]}(-)$ denotes the Knudsen-Mumford determinant functor.

The $S$-truncated equivariant $L$-function for $E/K$ is defined by
$$
\theta_{E/K,S}(s):=\sum_{\chi \in \widehat G}L_{K,S}(\chi^{-1},s)e_{\chi},
$$
where $L_{K,S}(\chi^{-1},s)$ is the $S$-truncated Artin $L$-series of $\chi^{-1}$. In particular, the leading term $\theta_{E/K,S}^\ast(0)$ at $s=0$ of $\theta_{E/K,S}(s)$ belongs to $\RR[G]^\times$ and so we may define a graded invertible $\ZZ[G]$-sublattice of $(\RR[G],0)$ by setting
\[
\Xi(E) := \theta_{E/K,S}^\ast(0)\cdot\vartheta_{\lambda_{E}}({\det}_{\ZZ[G]}( C_{E,S}))^{-1}.
\]
Then the conjecture ${\rm TNC}(h^0({\rm Spec}(E)),\cA)$ asserts that there is an equality in $(\RR[G]_\chi,0)$
$$
\cA\cdot \Xi(E)_\chi = (\cA,0).
$$

%
%


As a first step in the proof of Corollary \ref{cor main 2} we note that the argument of \cite[Th.~4.14]{bss2} implies an equality
\[
(\im(\eta_{E/K,S}^{V}),0) = {\rm Fitt}_{\ZZ[G]}^0(\ker(f))\cdot\Xi(E),
\]
where $f$ is the projection $H^1(C_{E,S}) \twoheadrightarrow X_{E,S} \twoheadrightarrow Y_{E,V}$ induced by (\ref{can ses2}).

In particular, since Hypothesis \ref{hyp chi}(iii) combines with (\ref{can ses2}) to imply $\ker(f)_\chi = {\rm Cl}(E)_\chi$, the last displayed equality combines with the second assertion of Corollary \ref{cor main 1}(i) to give an inclusion of graded ideals ${\rm Fitt}_{\ZZ[G]}^0({\rm Cl}(E))_\chi\cdot\Xi(E)_\chi \subset ({\rm Fitt}_{\ZZ[G]}^0({\rm Cl}(E))_\chi,0)$, and hence also an inclusion
\[
\cR_{E,\chi}\cdot\Xi(E)_\chi \subset (\cR_{E,\chi},0).
\]

To verify ${\rm TNC}(h^0({\rm Spec}(E)),\cR_{E,\chi})$ we need to prove that this inclusion is an equality, and by the argument used to prove \cite[Th. 5.1(ii)]{bss2} this can be deduced directly by using Nakayama's Lemma and the analytic class number formula for each intermediate field of $L/K$.

It now only remains only to prove the first assertion of Corollary \ref{cor main 2} and for this it is clearly enough to show that if ${\rm Cl}(L)_\chi$ vanishes, then
$\cR_{E,\chi} = \ZZ[G]_\chi$. This is true since the vanishing of ${\rm Cl}(L)_\chi$ combines with Hypothesis \ref{hyp chi}(iii) to imply that
 ${\rm Cl}(E)_\chi$ vanishes (again, by the same argument as in \cite[Th. 5.1(ii)]{bss2}).

\section{The refined class number formula}\label{sec rcnf}
In this section we state and prove a precise version of Corollary \ref{cor main 3}.

\subsection{Statement of the result}\label{formulation}

We first review the formulation of the (`$(p,\chi)$-component' of a certain special case of the) conjecture of Mazur and Rubin \cite[Conj. 5.2]{MRGm} and of the third author \cite[Conj. 3]{sano}.

To do this we recall some notation.
Let $F,\chi,L,\Delta,\Gamma,\cO$ be as in the previous section. We assume that $F$ contains the maximal $p$-extension $K(1)$ of $K$ inside its Hilbert class field $H_K$.
We use the following notations:
\begin{itemize}
\item $S:=S_\infty(K)\cup S_{\rm ram}(LF/K)$;
\item $\cP:=\{ \fq \notin S\cup S_p(K) \mid \text{$\fq$ splits completely in $LF H_K$}\}$;
\item $\cN:=\cN(\cP)=\{\text{square-free products of primes in $\cP$}\}$.
\end{itemize}
($\cN$ coincides with $\cN_0$ defined in \S \ref{sec koly}.)
For each $\fq$ in $\cP$ we write $K(\fq)/K$ for the maximal $p$-extension of $K$ inside its ray class field modulo $\fq$.

For each $\fn$ in $\cN$ we use the following notations:
\begin{itemize}
\item $\nu(\fn):=\# \{\fq \mid \fn\}$;
\item $K(\fn):=\prod_{\fq \mid \fn} K(\fq)$ (compositum);
\item $F(\fn):=F K(\fn)$;
\item $\cG_\fn:=\Gal(F(\fn)/K)\simeq \Gal(LF(\fn)/L)$;
\item $\cH_\fn:=\Gal(F(\fn)/F)\simeq \Gal(LF(\fn)/LF)$ (so $\Gamma=\Gal(F/K)\simeq \cG_\fn /\cH_\fn$);
\item $I_\fn:=\ker(\ZZ[\cH_\fn] \twoheadrightarrow \ZZ)$;
\item $S_\fn:=S \cup \{\fq \mid \fn\}$.
\end{itemize}
We write $V$ for the set of archimedean places of $K$ that split completely in $L$.  We set $r:=\# V$ and assume that $\# S >r >0$.
We set
$$V_\fn:=V\cup\{\fq \mid \fn\}$$
and note both that $\# V_\fn=r+\nu(\fn)$ and that all places in $V_\fn$ split completely in $LF$.

For the moment, we fix $\fn \in \cN$. We consider two Rubin-Stark elements for the data $(LF(\fn)/K,S_\fn,V)$ and $(LF/K,S_\fn,V_\fn)$ respectively:
$$\eta_{LF(\fn)/K,S_\fn}^{V} \in \RR \otimes_\ZZ {\bigwedge}_{\ZZ[\cG_\fn \times \Delta]}^r \cO_{LF(\fn),S_\fn}^\times,$$
$$\eta_{LF/K,S_\fn}^{V_\fn} \in \RR \otimes_\ZZ {\bigwedge}_{\ZZ[\Gamma\times \Delta]}^{r+\nu(\fn)} \cO_{LF,S_\fn}^\times.$$
(See Example \ref{ex1}(iii).)
Note that these elements depend on the choice of places of $LF(\fn)$ (or $LF$ for the latter) lying above places in
$S_\fn$, and also on the ordering of the places in $V_\fn$ (up to sign). See \cite[\S 5.2]{bks1}. These choices are supposed to be fixed at the beginning, and we will not mention them explicitly.

We consider $(p,\chi)$-components of Rubin-Stark elements:
$$c_\fn:=\eta_{LF(\fn)/K,S_\fn}^{V,\chi} \in \left(\RR \otimes_\ZZ {\bigwedge}_{\ZZ[\cG_\fn \times \Delta]}^r \cO_{LF(\fn),S_\fn}^\times\right)_\chi \simeq \RR \otimes_\ZZ {\bigwedge}_{\cO[\cG_\fn]}^r (\cO_{LF(\fn),S_\fn}^\times)_\chi,$$
$$\epsilon_\fn:=\eta_{LF/K,S_\fn}^{V_\fn,\chi} \in \left(\RR \otimes_\ZZ {\bigwedge}_{\ZZ[\Gamma\times \Delta]}^{r+\nu(\fn)} \cO_{LF,S_\fn}^\times\right)_\chi\simeq \RR \otimes_\ZZ {\bigwedge}_{\cO[\Gamma]}^{r+\nu(\fn)} (\cO_{LF,S_\fn}^\times)_\chi.$$

We now suppose that modules
$$
U_{F(\fn),S_{\fn}}:=(\cO_{LF(\fn),S_\fn}^\times)_\chi \text{ and }
U_{F,S_{\fn}}:=(\cO_{LF,S_\fn}^\times)_\chi
$$
are $\cO$-free, and the ($(p,\chi)$-component of the) Rubin-Stark conjecture is true so that
$$
c_\fn \in {\bigcap}_{\cO[\cG_\fn]}^r U_{F(\fn),S_\fn} \text{ and }\epsilon_\fn \in {\bigcap}_{\cO[\Gamma]}^{r+\nu(\fn)} U_{F,S_\fn}.
$$
We recall that, under Hypothesis \ref{hyp chi}, the $\cO$-modules $U_{F(\fn),S_\fn}$ and $U_{F,S_\fn}$ are both free (cf. Remark \ref{rem hyp}).

We shall now formulate the conjecture of Mazur-Rubin and Sano, which asserts that a certain congruence relation holds between Rubin-Stark elements $c_\fn$ and $\epsilon_\fn$.

We review some necessary constructions. First, there is a natural injection
\begin{eqnarray}\label{iota}
\iota_\fn: {\bigcap}_{\cO[\Gamma]}^{r} U_{F,S_\fn} \otimes_\ZZ I_\fn^{\nu(\fn)}/I_\fn^{\nu(\fn)+1} \hookrightarrow {\bigcap}_{\cO[\cG_\fn]}^{r} U_{F(\fn),S_\fn} \otimes_\ZZ \ZZ[\cH_\fn]/ I_\fn^{\nu(\fn)+1}.
\end{eqnarray}
This map is induced by the natural inclusion $I_\fn^{\nu(\fn)}/I_\fn^{\nu(\fn)+1} \hookrightarrow \ZZ[\cH_\fn]/I_\fn^{\nu(\fn)+1}$ and the homomorphism
\begin{eqnarray}\label{incl}
{\bigcap}_{\cO[\Gamma]}^r U_{F,S_\fn} \to {\bigcap}_{\cO[\cG_\fn]}^r U_{F(\fn),S_\fn},
\end{eqnarray}
that arises as the linear dual of the restriction map
\begin{eqnarray*}
{\bigwedge}_{\cO[\cG_\fn]}^r \Hom_{\cO[\cG_\fn]}(U_{F(\fn),S_\fn},\cO[\cG_\fn]) &\simeq& {\bigwedge}_{\cO[\cG_\fn]}^r \Hom_{\cO}(U_{F(\fn),S_\fn},\cO)  \\
&\to& {\bigwedge}_{\cO[\Gamma]}^r \Hom_{\cO}(U_{F,S_\fn},\cO) \simeq {\bigwedge}_{\cO[\Gamma]}^r \Hom_{\cO[\Gamma]}(U_{F,S_\fn},\cO[\Gamma]).
\end{eqnarray*}
The injectivity of $\iota_\fn$ follows from the fact that the cokernel of (\ref{incl}) is torsion-free. (Compare \cite[Lem. 2.11]{sano}.)

Next, there is a `reciprocity homomorphism'
$$
{\rm Rec}_\fn: {\bigcap}_{\cO[\Gamma]}^{r+\nu(\fn)} U_{F,S_\fn} \to {\bigcap}_{\cO[\Gamma]}^r U_{F,S_\fn}\otimes_\ZZ I_\fn^{\nu(\fn)}/I_\fn^{\nu(\fn)+1}.
$$
To construct this, we first define a map $\varphi_\fq^\fn$ for each $\fq \in \cP$ as follows:
$$
\varphi_\fq^\fn: LF^\times \to \ZZ[ \Delta \times \Gamma]\otimes_\ZZ I_\fn/I_\fn^2  ;\ a \mapsto \sum_{\sigma \in  \Delta \times \Gamma}\sigma^{-1} \otimes ({\rm rec}_\fQ(\sigma a)-1),
$$
where $\fQ$ is the fixed place of $LF$ lying above $\fq$ and ${\rm rec}_\fQ: LF^\times  \to \Gal(LF(\fn)/LF)=\cH_\fn$ is the local reciprocity map at $\fQ$. This map induces a map
$$
U_{F,S_\fn} \to \cO[\Gamma] \otimes_\ZZ I_\fn/I_\fn^2,
$$
which is also denoted by $\varphi_\fq^\fn$. Then ${\rm Rec}_\fn$ is defined in \cite[Prop. 2.7]{sano} to be a canonical extension of the map
$$
{\bigwedge}_{\fq \mid \fn}\varphi_\fq^\fn:  {\bigwedge}_{\cO[\Gamma]}^{r+\nu(\fn)} U_{F,S_\fn} \to {\bigwedge}_{\cO[\Gamma]}^r U_{F,S_\fn} \otimes_\ZZ I_\fn^{\nu(\fn)}/I_\fn^{\nu(\fn)+1}.
$$

The conjecture is formulated as follows.

\begin{conjecture}\label{mrs} In the module ${\bigcap}_{\cO[\cG_\fn]}^r U_{F(\fn),S_\fn} \otimes_\ZZ \ZZ[\cH_\fn] /I_\fn^{\nu(\fn)+1}$ one has
$$
\sum_{\sigma \in \cH_\fn} \sigma c_\fn \otimes \sigma^{-1} = \iota_\fn \left( {\rm Rec}_\fn(\epsilon_\fn)\right).
$$
\end{conjecture}

\begin{remark} \label{rem1}
Since $c_1=\epsilon_1$, Conjecture \ref{mrs} is trivially true in the case $\fn=1$.
\end{remark}

\begin{remark}\label{rem2} The conjecture formulated in \cite[Conj.~5.2]{MRGm} and \cite[Conj.~3]{sano} is actually much more general than the above in that it is formulated for all sets of data of the form $(L'/L/K,S,T,V,V')$, where $L'/L/K$ are finite extensions of global fields such that $L'/K$ is abelian and $S, T, V,V'$ are certain sets of places of $K$. Conjecture \ref{mrs} is simply the $(p,\chi)$-component of the general conjecture applied to the data  $(LF(\fn)/LF/K,S_\fn,\emptyset, V,V_\fn)$. We recall that a further refinement of the general conjecture is given in \cite[Conj. 5.4]{bks1}.
\end{remark}

\begin{remark}\label{rem3}
One can slightly extend the formulation of Conjecture \ref{mrs} in our setting as follows.
Let $\fm$ be any square-free product of primes $\fq$ that do not belong to $S\cup S_p(K)$ (and are also not required to belong to $\cN$). Let $\fm_+$ be the product of prime divisors $\fq$ of $\fm$ that split completely in $LF$ (so that $\fm_+$ belongs to $\cN$). Then, according to Remark~\ref{rem2}, one can formulate Conjecture \ref{mrs} for the data $(LF(\fm)/LF/K,S_\fm,\emptyset, V,V_{\fm_+}).$

However, one can show that Conjecture \ref{mrs} for this data is implied by Conjecture \ref{mrs} for the data $(LF(\fm_+)/LF/K,S_{\fm_+},\emptyset,V,V_{\fm_+})$ and so such a generalization contains no new information.
\end{remark}

\begin{remark}\label{rem4} The formulation of
Conjecture \ref{mrs} is motivated by a conjecture of Darmon from \cite{D}. Darmon's conjecture is obtained by specializing the form of Conjecture \ref{mrs} discussed in Remark \ref{rem3} to the data $(L(n)/L/\QQ,\{\infty\}\cup \{\ell \mid nf\},\emptyset,\{\infty\}, \{\infty\}\cup \{\ell \mid n_+\}),$ where $L$ is a real quadratic field of conductor $f$, $n$ is a square-free product of primes not dividing $fp$, $L(n)/L$ is the maximal $p$-extension inside $L(\mu_n)$, and $\infty$ denotes the infinite place of $\QQ$. For details of this deduction see \cite[\S 6.1]{bks1},  where the case $p=2$ is also treated.\end{remark}





We now state the result. 
Then the following result is a precise version of Corollary~\ref{cor main 3}.

\begin{theorem}\label{main} Conjecture \ref{mrs} is valid for every ideal $\fn$ in $\cN$ whenever all of the following conditions are satisfied;
\begin{itemize}
\item the Rubin-Stark conjecture is valid for all abelian extensions of $K$;
\item Hypothesis~\ref{hyp chi};
\item either \begin{itemize}
\item there exists a non-trivial $\ZZ_p$-power extension $K_\infty$ of $K$ in which no finite places split completely and $S_{\rm ram}(K_\infty/K) \subset S_{\rm ram}(LF/K)$, or
\item $L \not\subset K(\mu_{p})$.
\end{itemize}
\end{itemize}
\end{theorem}


After several preliminary, and quite technical, sections this result is proved in \S\ref{proof main}.

\subsection{Strategy of the proof}\label{idea} The strategy that we use to prove Theorem \ref{main} is a natural development of an approach used by Mazur and Rubin in \cite{MR} to prove (the `non-$2$-part' of) Darmon's Conjecture. In fact, if one specializes Theorem \ref{main} to the setting of Remark \ref{rem4}, then all assumptions in Theorem \ref{main} are unconditionally satisfied and the result simply recovers the main result of \cite{MR}.

For the convenience of the reader, in this subsection we shall sketch the proof of Theorem \ref{main}. To do this we first set some notations.

We write $T$ for the representation $\cO(1)\otimes \chi^{-1}$ discussed in Remark \ref{twisted rep}. Then, setting 
$$
\Sigma:=S\cup S_p(K) \text{ and } \Sigma_\fn:=S_\fn \cup S_p(K),
$$
there is a natural identification
$$
H^1(\cO_{F(\fn),\Sigma_\fn},T) \simeq (\cO_{LF(\fn),\Sigma_\fn}^\times)_\chi.
$$
There are therefore natural embeddings
$$
U_{F(\fn),S_\fn} \hookrightarrow H^1(\cO_{F(\fn),\Sigma_\fn},T) \text{ and }
U_{F, S_\fn} \hookrightarrow H^1(\cO_{F,\Sigma_\fn},T),
$$
via which we can regard
$$c_\fn \in {\bigcap}_{\cO[\cG_\fn]}^r H^1(\cO_{F(\fn),\Sigma_\fn},T) \text{ and }\epsilon_\fn \in {\bigcap}_{\cO[\Gamma]}^{r+\nu(\fn)} H^1(\cO_{F,\Sigma_\fn},T).$$

One sees that the induced map
$${\bigcap}_{\cO[\cG_\fn]}^r U_{F(\fn)} \otimes_\ZZ \ZZ[\cH_\fn]/I_\fn^{\nu(\fn)+1} \to {\bigcap}_{\cO[\cG_\fn]}^r H^1(\cO_{F(\fn),\Sigma_\fn},T) \otimes_\ZZ \ZZ[\cH_\fn]/I_\fn^{\nu(\fn)+1}$$
is injective, so Conjecture \ref{mrs} is equivalent to the assertion that the equality holds in the latter group. (We use Galois cohomology only for notational convenience.)

In the following, we set
$$\theta_\fn:=\sum_{\sigma \in \cH_\fn} \sigma c_\fn \otimes \sigma^{-1 } \in {\bigcap}_{\cO[\cG_\fn]}^r H^1(\cO_{F(\fn),\Sigma_\fn},T)\otimes_\ZZ \ZZ[\cH_\fn]/I_\fn^{\nu(\fn)+1}.$$


We can now give a brief idea of the proof of Theorem \ref{main}. First, we observe that the systems
$$(\theta_\fn)_{\fn \in \cN} \text{ and }({\rm Rec}_\fn(\epsilon_\fn))_{\fn \in \cN}$$
are essentially Kolyvagin systems (see Propositions \ref{unram ks} and \ref{prop reg}).
By Theorem \ref{th main}(ii), we know that the module of Kolyvagin systems is free of rank one (over $\cO[\Gamma]$ in our case), there exists a basis $\kappa$ of the module and we can write
$$(\theta_\fn)_\fn = a \cdot \kappa \text{ and }({\rm Rec}_\fn(\epsilon_\fn))_\fn =b \cdot \kappa$$
with some $a,b \in \cO[\Gamma]$. So it reduces to show that $a=b$. However, since the conjecture is trivially true when $\fn=1$ (see Remark \ref{rem1}), we have
$$(a-b)\cdot \kappa_1=0.$$
Then it is easy to show that ${\rm Ann}_{\cO[\Gamma]}(\kappa_1)=0$, which implies $a=b$ and hence completes the proof of Theorem \ref{main}.

\subsection{Kolyvagin derivatives}\label{koly der}


In this subsection, we relate the element
$$  \theta_\fn=\sum_{\sigma \in \cH_\fn}\sigma c_\fn \otimes \sigma^{-1}$$
with the `Kolyvagin derivative'.
 
The construction given here is valid for a general Euler system $c$ in ${\rm ES}_r(\cK/K,S,\chi)$ (by replacing $c_\fn$ by $c_{F(\fn)}$), where we take $\cK$ to be sufficiently large so that it contains $F$, the maximal $p$-extension inside the ray class field modulo $\fq$ for all but finitely many $\fq$, and the maximal abelian pro-$p$ extension of $K$ unramified outside $p$. In particular, in this way one obtains an explicit construction of the derivative homomorphism $\cD_{F,r}$ in Theorem \ref{th main}(i).

We give some preliminaries. As in \S \ref{sec koly}, for $\fq \in \cP$, set
$$G_\fq:=\Gal(K(\fq)/K(1)),$$
and for $\fn \in \cN$, set
$$G_\fn :=\bigotimes_{\fq \mid \fn} G_\fq.$$
Note that
$$\cH_\fn(:=\Gal(F(\fn)/F)) \simeq \Gal(K(\fn)/K(1))\simeq \prod_{\fq \mid \fn} G_\fq.$$
From this, if $\fd$ is a divisor of $\fn$, one can regard $\cH_\fd$ as both a subgroup and a quotient of $\cH_\fn$. Let
$$\pi_\fd: \cH_\fn \to \cH_\fd$$
be the natural projection map. This map induces a map
$$\Z[\cH_\fn] \to \ZZ[\cH_\fd]\subset \ZZ[\cH_\fn],$$
which is also denoted by $\pi_\fd$. We define
$$s_\fn:\ZZ[\cH_\fn] \to \ZZ[\cH_\fn] ; \ a \mapsto \sum_{\fd \mid \fn} (-1)^{\nu(\fn/\fd)}\pi_\fd(a),$$
where $\fd$ runs over all divisors of $\fn$ (including $1$). This map induces endomorphisms of
$$I_\fn^a/I_\fn^{a+1}, \ X\otimes_\ZZ \ZZ[\cH_\fn] \text{ (for any module $X$)}, \ \text{etc.,}$$
which we denote also by $s_\fn$.

We fix a generator $\sigma_\fq$ of $G_\fq$ for each $\fq \in \cP$. Let
$$D_\fq:=\sum_{i=1}^{\# G_\fq -1}i \sigma_\fq^{i} \in \ZZ[G_\fq], \ D_\fn:=\prod_{\fq \mid \fn}D_\fq \in \ZZ[\cH_\fn]$$
be Kolyvagin's derivative operators.

We have the following algebraic lemma.

\begin{lemma}\label{lem1} For each ideal $\fn$ in $\cN$ set $Q_\fn := I_\fn^{\nu(\fn)}/I_\fn^{\nu(\fn)+1}$.
\begin{itemize}
\item[(i)] The image of the map $s_\fn :\ZZ[\cH_\fn]\to \ZZ[\cH_\fn]$ is contained in $I_\fn^{\nu(\fn)}$.
\item[(ii)] The image of the map $\ZZ[\cH_\fn]/I_\fn^{\nu(\fn)+1} \to Q_\fn$ that is induced by $s_\fn$ is contained in $\left\langle \prod_{\fq \mid \fn}(\sigma_\fq-1)\right\rangle_\ZZ$.
\item[(iii)] There is a natural isomorphism
$$G_\fn \xrightarrow{\sim} \left\langle \prod_{\fq \mid \fn}(\sigma_\fq-1)\right\rangle_\ZZ \subset Q_\fn; \ \bigotimes_{\fq \mid \fn}\sigma_\fq \mapsto \prod_{\fq \mid \fn}(\sigma_\fq-1).$$
\item[(iv)] The map
$$ Q_\fn \to \left\langle \prod_{\fq \mid \fn}(\sigma_\fq-1)\right\rangle_\ZZ \simeq G_\fn$$
induced by $s_\fn$ gives a splitting of the injection $G_\fn \hookrightarrow Q_\fn$ in (iii).
\item[(v)] Let $X$ be a $\ZZ[\cH_\fn]$-module and $x \in X$. Denote the map
$$X \otimes_\ZZ \ZZ[\cH_\fn]/I_\fn^{\nu(\fn)+1} \to X \otimes_\ZZ  \left\langle \prod_{\fq \mid \fn}(\sigma_\fq-1)\right\rangle_\ZZ \simeq X \otimes_\ZZ G_\fn$$
induced by $s_\fn$ also by the same notation. Then we have
$$s_\fn\left(\sum_{\sigma \in \cH_\fn}\sigma x \otimes \sigma^{-1}\right)=(-1)^{\nu(\fn)}D_\fn x \otimes \bigotimes_{\fq \mid \fn} \sigma_\fq.$$
\end{itemize}

\end{lemma}
\begin{proof}
See \cite[Lem. 4.27 and 4.28]{sbA}.
\end{proof}

For $\fq \in \cP$, we set $P_\fq:=1-{\rm Fr}_\fq^{-1}.$ Here ${\rm Fr}_\fq$ denotes the (arithmetic) Frobenius element at $\fq$, which is regarded as an element of $\cH_\fn$ for any $\fn \in \cN$, via the injection $\cH_{\fn/\fq} \hookrightarrow \cH_\fn$ when $\fq \mid \fn$. Thus $P_\fq$ is regarded as an element of $I_\fn/I_\fn^2$ for any $\fn \in \cN$. We denote by
${\rm Fr}_\fq^\fn$ the Frobenius ${\rm Fr}_\fq$ regarded as an element of $\cH_\fn$. Similarly, we denote by $P_\fq^\fn$ the element $P_\fq$ regarded as an element of $I_\fn/I_\fn^2$.

\begin{lemma}\label{lem0} For each ideal $\fd$ in $\cN$ set $Q_\fd := I_\fd^{\nu(\fd)}/I_\fd^{\nu(\fd)+1}$.
\begin{itemize}
\item[(i)] For $\fd\mid \fn$, there is a natural injection
\begin{equation}\label{inj0} {\bigcap}_{\cO[\cG_\fd]}^r H^1(\cO_{F(\fd),\Sigma_\fd},T) \otimes_\ZZ Q_\fd \hookrightarrow {\bigcap}_{\cO[\cG_\fn]}^r H^1(\cO_{F(\fn),\Sigma_\fn},T) \otimes_\ZZ Q_\fd.\end{equation}
\item[(ii)] We have
$$  \theta_\fn=\sum_{\sigma \in \cH_\fn}\sigma c_\fn \otimes \sigma^{-1}  \in {\bigcap}_{\cO[\cG_\fn]}^r H^1(\cO_{F(\fn),\Sigma_\fn},T) \otimes_\ZZ Q_\fn$$
and
$$  \theta_\fn=s_\fn(  \theta_\fn)-\sum_{\fd \mid \fn, \ \fd \neq \fn}(-1)^{\nu(\fn/\fd)}  \theta_\fd \prod_{\fq \mid \fn/\fd}P_\fq^\fd ,$$
where we use (\ref{inj0}) to regard $\theta_\fd$ as an element of ${\bigcap}_{\cO[\cG_\fn]}^r H^1(\cO_{F(\fn),\Sigma_\fn},T) \otimes_\ZZ Q_\fd$.
\end{itemize}
\end{lemma}

\begin{proof} The construction of the injection in claim (i) is the same as that of $\iota_\fn$ in (\ref{iota}).

To prove claim (ii) we note that, by the definition of $s_\fn$, one has
$$s_\fn(\theta_\fn)=\sum_{\fd \mid \fn}(-1)^{\nu(\fn/\fd)}\pi_\fd(\theta_\fn).$$
Setting ${\N}_{\fn/\fd}:=\sum_{\sigma \in \cH_{\fn/\fd}}\sigma$, we compute
\begin{eqnarray*}
\pi_\fd(\theta_\fn)&=&\pi_\fd\left( \sum_{\sigma \in \cH_\fn}\sigma c_\fn \otimes \sigma^{-1}\right) \\
&=&\sum_{\sigma \in \cH_\fn}\sigma c_\fn \otimes \pi_\fd(\sigma)^{-1} \\
&=& \sum_{\sigma \in \cH_\fd} \sigma {\N_{\fn/\fd}}c_\fn \otimes \sigma^{-1} \\
&=& \sum_{\sigma \in \cH_\fd}\sigma\left( \prod_{\fq \mid \fn/\fd}P_\fq^\fd\right) c_\fd \otimes \sigma^{-1} \\
&=&\sum_{\sigma \in \cH_\fd}\sigma c_\fd \otimes \sigma^{-1} \prod_{\fq \mid\fn/\fd}P_\fq^\fd \\
&=&\theta_\fd \prod_{\fq \mid \fn/\fd}P_\fq^\fd,
\end{eqnarray*}
where the fourth equality follows from the well-known `norm relation'
$${\N}_{\fn/\fd}c_\fn =\left(\prod_{\fq \mid \fn/\fd}P_\fq^\fd \right) c_\fd \text{ in }{\bigcap}_{\cO[\cG_\fn]}^r H^1(\cO_{F(\fn),\Sigma_\fn},T).$$
(See \cite[Prop. 3.6 and Lem. 4.4]{sanotjm}, for example.)
Thus we have
$$s_\fn(\theta_\fn)=\sum_{\fd \mid \fn}(-1)^{\nu(\fn/\fd)}\pi_\fd(\theta_\fn)=\sum_{\fd \mid \fn}(-1)^{\nu(\fn/\fd)} \theta_\fd \prod_{\fq \mid \fn/\fd}P_\fq^\fd.$$
The assertion follows from this, by using Lemma \ref{lem1}(i) and induction on $\nu(\fn)$.
\end{proof}

We set $M_\fn:=\# G_\fn,$ with $M_1$ understood to be $0$. Since we fixed a generator $\sigma_\fq \in G_\fq$ for each $\fq \in \cP$, we have an identification
$$G_\fn \xrightarrow{\sim} \ZZ/M_\fn \ZZ ; \ \bigotimes_{\fq \mid \fn}\sigma_\fq \mapsto 1.$$
We set $\cO_\fn:=\cO/M_\fn \cO$ and $A_\fn:=T/M_\fn T.$

\begin{lemma}\label{lem2}\
\begin{itemize}
\item[(i)] There are natural injections
\begin{eqnarray}\label{inj1}
{\bigcap}_{\cO[\cG_\fn]}^r H^1(\cO_{F(\fn),\Sigma_\fn},T) \otimes_\ZZ G_\fn \hookrightarrow {\bigcap}_{\cO_\fn[\cG_\fn]}^r H^1(\cO_{F(\fn),\Sigma_\fn},A_\fn)\otimes_\ZZ G_\fn
\end{eqnarray}
and
\begin{eqnarray}\label{inj2}
{\bigcap}_{\cO_\fn[\Gamma]}^r H^1(\cO_{F,\Sigma_\fn},A_\fn)\otimes_\ZZ G_\fn \hookrightarrow {\bigcap}_{\cO_\fn[\cG_\fn]}^r H^1(\cO_{F(\fn),\Sigma_\fn},A_\fn)\otimes_\ZZ G_\fn.
\end{eqnarray}
\item[(ii)] The image of $  \theta_\fn=\sum_{\sigma \in \cH_\fn}\sigma c_\fn \otimes \sigma^{-1}$ under the map
\begin{eqnarray*}
{\bigcap}_{\cO[\cG_\fn]}^r H^1(\cO_{F(\fn),\Sigma_\fn},T) \otimes_\ZZ Q_\fn&\xrightarrow{s_\fn} &{\bigcap}_{\cO[\cG_\fn]}^r H^1(\cO_{F(\fn),\Sigma_\fn},T) \otimes_\ZZ G_\fn \\
&\stackrel{(\ref{inj1})}{\hookrightarrow}& {\bigcap}_{\cO_\fn[\cG_\fn]}^r H^1(\cO_{F(\fn),\Sigma_\fn},A_\fn) \otimes_\ZZ G_\fn
\end{eqnarray*}
lies in ${\bigcap}_{\cO_\fn[\Gamma]}^r H^1(\cO_{F,\Sigma_\fn},A_\fn) \otimes_\ZZ G_\fn$ (i.e., the image of the injection (\ref{inj2})).
\end{itemize}

\end{lemma}

\begin{proof}
(i) The construction of (\ref{inj2}) is the same as that of $\iota_\fn$ in (\ref{iota}). (Actually, we have
$$\left( {\bigcap}_{\cO_\fn[\cG_\fn]}^r H^1(\cO_{F(\fn),\Sigma_\fn},A_\fn)\right)^{\cH_\fn}={\bigcap}_{\cO_\fn[\Gamma]}^r H^1(\cO_{F,\Sigma_\fn},A_\fn).$$
See \cite[\S 4.3.1]{sbA}.) We give a construction of (\ref{inj1}). For simplicity, set
$$R:=\cO[\cG_\fn], \ M:=M_\fn, \ A:=A_\fn=T/MT, \ H(-):=H^1(\cO_{F(\fn),\Sigma_\fn},-).$$
By identifying $G_\fn=\ZZ/M\ZZ$, it is sufficient to construct an injection
$$\left( {\bigcap}_R^r H(T) \right) \otimes_\ZZ \ZZ/M\ZZ \hookrightarrow {\bigcap}_{R/MR}^r H(A).$$
First, we note that there is a natural injection $H(T)\otimes_\ZZ \ZZ/M\ZZ \hookrightarrow H(A)$ which in turn induces a surjection
$$\Hom_{R/MR}(H(A),R/MR) \to \Hom_{R/MR}(H(T)\otimes_\ZZ \ZZ/M\ZZ ,R/MR).$$
Since $H(T)$ is $\cO$-free, the latter module is isomorphic to $\Hom_R(H(T),R)\otimes_\ZZ \ZZ/M\ZZ$, and thus we obtain a surjection
$$\Hom_{R/MR}(H(A),R/MR) \to \Hom_R(H(T),R)\otimes_\ZZ \ZZ/M\ZZ.$$
This also gives a surjection
$${\bigwedge}_{R/MR}^r \Hom_{R/MR}(H(A),R/MR) \to \left({\bigwedge}_R^r\Hom_R(H(T),R) \right)\otimes_\ZZ \ZZ/M\ZZ.$$
The desired injection is obtained by composing the $R/MR$-dual of this surjection with the natural injection
$$\left( {\bigcap}_R^r H(T) \right)\otimes_\ZZ \ZZ/M\ZZ \hookrightarrow \Hom_R\left({\bigwedge}_R^r \Hom_R(H(T),R),R/MR\right).$$

(ii) By \cite[Lem. 6.9]{bss}, we know that the image of $D_\fn c_\fn \otimes \bigotimes_{\fq \mid \fn}\sigma_\fq$ under the map
(\ref{inj1}) lies in ${\bigcap}_{\cO_\fn[\Gamma]}^r H^1(\cO_{F,\Sigma_\fn},A_\fn) \otimes_\ZZ G_\fn$. Since we have
$$s_\fn(\theta_\fn)=\pm D_\fn c_\fn \otimes \bigotimes_{\fq \mid \fn}\sigma_\fq$$
by Lemma \ref{lem1}(v), the claim follows.
\end{proof}

By Lemma \ref{lem2}(ii), we can regard
$$s_\fn(  \theta_\fn) \in  {\bigcap}_{\cO_\fn[\Gamma]}^r H^1(\cO_{F,\Sigma_\fn},A_\fn)\otimes_\ZZ G_\fn.$$


We define an endomorphism
$$\Psi \in \End\left( \prod_{\fn \in \cN} {\bigcap}_{\cO_\fn[\Gamma]}^r H^1(\cO_{F,\Sigma_\fn},A_\fn)\otimes_\ZZ G_\fn \right)$$
by
$$\Psi((a_\fn)_\fn):=\left( (-1)^{\nu(\fn)}\sum_{\tau \in \mathfrak{S}(\fn)} {\rm sgn}(\tau) a_{\fd_\tau} \otimes \bigotimes_{\fq \mid \fn/\fd_\tau} {\rm Fr}_{\tau(\fq)}^\fq \right)_\fn,$$
where
\begin{itemize}
\item $\mathfrak{S}(\fn)$: the set of permutations of the set $\{\fq \mid \fn\}$,
\item $\fd_{\tau}$: the product of $\fq \mid \fn$ fixed by $\tau$,
\item ${\rm Fr}_{\tau(\fq)}^\fq$: the Frobenius element at $\tau(\fq)$, regarded as an element of $G_\fq (\simeq \cH_\fq)$.
\end{itemize}

\begin{lemma}\label{lempsi}
$\Psi$ is injective.
\end{lemma}
\begin{proof}
Let $a=(a_\fn)_\fn \in \prod_{\fn \in \cN} {\bigcap}_{\cO_\fn[\Gamma]}^r H^1(\cO_{F,\Sigma_\fn},A_\fn)\otimes_\ZZ G_\fn$ and suppose that $\Psi(a)=0$. We show $a_\fn=0$ by induction on $\nu(\fn)$. When $\nu(\fn)=0$, i.e., $\fn=1$, we have
$$a_1=\Psi(a)_1=0.$$
When $\nu(\fn)>0$, we have
$$(-1)^{\nu(\fn)}\Psi(a)_\fn = a_\fn + \sum_{\tau \in \mathfrak{S}(\fn), \ \tau \neq {\rm id}} {\rm sgn}(\tau) a_{\fd_\tau} \otimes \bigotimes_{\fq \mid \fn/\fd_\tau} {\rm Fr}_{\tau(\fq)}^\fq.$$
By the induction hypothesis, the second term on the right hand side vanishes. Thus we have $a_\fn=(-1)^{\nu(\fn)}\Psi(a)_\fn=0$.
\end{proof}

Let
$${\rm KS}_r(F/K,\chi)=\varprojlim_m {\rm KS}_r(\cP_m)_m \subset \prod_m \prod_{\fn \in \cN_m} {\bigcap}_{\cO_m[\Gamma]}^r \cS(\fn)_m \otimes_\ZZ G_\fn$$
be the module of Kolyvagin systems of rank $r$ defined in \S \ref{limits section}. We have a natural embedding
\begin{eqnarray*}\label{emb koly}
{\rm KS}_r(F/K,\chi) \hookrightarrow \prod_{\fn \in \cN} {\bigcap}_{\cO_\fn[\Gamma]}^r H^1(\cO_{F,\Sigma_\fn},A_\fn)\otimes_\ZZ G_\fn
\end{eqnarray*}
defined by
$$x=(x_m)_m=((x_{m,\fn})_\fn)_m \mapsto (x_{m_\fn,\fn})_\fn,$$
where 
we define $m_\fn \in \ZZ$ by $p^{m_\fn}=M_\fn$ and $x_{m_\fn,\fn} \in {\bigcap}_{\cO_\fn[\Gamma]}^r \cS(\fn)_{m_\fn} \otimes_\ZZ G_\fn$ is regarded as an element of ${\bigcap}_{\cO_\fn[\Gamma]}^r H^1(\cO_{F,\Sigma_\fn},A_\fn)\otimes_\ZZ G_\fn$ via the natural embedding $\cS(\fn)_{m_\fn} \hookrightarrow H^1(\cO_{F,\Sigma_\fn},A_\fn)$. (Note that if $\fn=1$, then $m_\fn$ is not defined and so $x_{m_1,1}$ is understood to be the element $(x_{m,1})_m$ of $\varprojlim_m {\bigcap}_{\cO_m[\Gamma]}^r H^1(\cO_{F,\Sigma},T/p^mT)={\bigcap}_{\cO[\Gamma]}^r H^1(\cO_{F,\Sigma},T)$.)


\begin{proposition}\label{unram ks} The system $\kappa:=\Psi((s_\fn(  \theta_\fn))_\fn)$ belongs to ${\rm KS}_r(F/K,\chi).$
\end{proposition}

\begin{proof} This follows from Theorem \ref{th main}(i) since the homomorphism $\cD_{F,r}$ sends $c$ to $\kappa$.
\end{proof}

\subsection{Regulator Kolyvagin systems}

We first define the module of Stark systems `over $\cO[\Gamma]$'. The construction is similar to that defined `over $\cO_m[\Gamma]$' in \S \ref{sec koly}.


For $\fm,\fn\in \cN$ with $\fn \mid \fm$, we define a map
$$v_{\fm,\fn}: {\bigcap}_{\cO[\Gamma]}^{r+\nu(\fm)} H^1(\cO_{F,\Sigma_\fm},T) \to {\bigcap}_{\cO[\Gamma]}^{r+\nu(\fn)}H^1(\cO_{F,\Sigma_\fn},T)$$
as follows. First, for $\fq \in \cP$
let $v_\fq : LF^\times \to \ZZ[ \Delta \times \Gamma]$ be the map defined in (\ref{def v}).
This map induces
$$v_\fq : H^1(\cO_{F,\Sigma_\fm},T)\simeq (\cO_{LF,\Sigma_\fm}^\times)_\chi \to \cO[\Gamma].$$
Then we set
$$v_{\fm,\fn}:=\pm {\bigwedge}_{\fq \mid \fm/\fn}v_\fq$$
and note that the argument of \cite[Prop. 3.6]{sano} shows that the image of this map is contained in ${\bigcap}_{\cO[\Gamma]}^{r+\nu(\fn)}H^1(\cO_{F,\Sigma_\fn},T)$.

We can consider the inverse limit
$${\rm SS}_r(T):=\varprojlim_{\fn \in \cN} {\bigcap}_{\cO[\Gamma]}^{r+\nu(\fn)} H^1(\cO_{F,\Sigma_\fn},T)$$
by taking $v_{\fm,\fn}$ as the transition map.
It is well-known that
$$\epsilon:=(\epsilon_\fn)_\fn \in {\rm SS}_r(T).$$
(See loc. cit.)
We construct an `algebraic regulator' homomorphism 
$$\cR: {\rm SS}_r(T) \to \prod_{\fn \in \cN} {\bigcap}_{\cO[\Gamma]}^r H^1(\cO_{F,\Sigma_\fn},T)\otimes_\ZZ G_\fn$$
as follows. Let
$$\varphi_\fq^{\fn}: H^1(\cO_{F,\Sigma_\fm},T)\simeq (\cO_{LF,\Sigma_\fm}^\times)_\chi \to \cO[\Gamma]\otimes_\ZZ I_\fn/I_\fn^2$$
be the map constructed in \S \ref{formulation} (for any $\fq \in \cP$ and $\fm,\fn\in\cN$). We define $\varphi_\fq^{\rm fs}$ by
$$\varphi_\fq^{\rm fs}: H^1(\cO_{F,\Sigma_\fn},T) \xrightarrow{\varphi_\fq^\fq} \cO[\Gamma]\otimes_\ZZ I_\fq/I_\fq^2 \simeq \cO[\Gamma] \otimes_\ZZ G_\fq,$$
where the last isomorphism is induced by
$$I_\fq/I_\fq^2 \xrightarrow{\sim} G_\fq; \ \sigma_\fq-1 \mapsto \sigma_\fq.$$
For $a=(a_\fn)_\fn \in {\rm SS}_r(T)$, we define the regulator by
$$\cR(a)_\fn:=\left({\bigwedge}_{\fq \mid \fn}\varphi_\fq^{\rm fs} \right)(a_\fn).$$


We now relate $({\rm Rec}_\fn(\epsilon_\fn))_\fn$ with $\cR(\epsilon)$. Recall that
$${\rm Rec}_\fn: {\bigcap}_{\cO[\Gamma]}^{r+\nu(\fn)} H^1(\cO_{F,\Sigma_\fn},T) \to {\bigcap}_{\cO[\Gamma]}^r H^1(\cO_{F,\Sigma_\fn},T)\otimes_\ZZ Q_\fn$$
is defined by ${\bigwedge}_{\fq \mid \fn}\varphi_\fq^\fn$. We shall use the projector
$$s_\fn: Q_\fn \to G_\fn$$
and the endomorphism
$$\Psi \in \End\left( \prod_{\fn \in \cN} {\bigcap}_{\cO[\Gamma]}^r H^1(\cO_{F,\Sigma_\fn},T)\otimes_\ZZ G_\fn \right)$$
defined in the previous subsection. (Note that we can regard
$${\bigcap}_{\cO[\Gamma]}^r H^1(\cO_{F,\Sigma_\fn},T)\otimes_\ZZ G_\fn \subset {\bigcap}_{\cO_\fn[\Gamma]}^r H^1(\cO_{F,\Sigma_\fn},A_\fn)\otimes_\ZZ G_\fn$$
via the injection constructed in the same way as (\ref{inj1}) in Lemma \ref{lem2}(i).)

\begin{proposition}\label{prop reg} $\Psi\left( \left(s_\fn\left({\rm Rec}_\fn(\epsilon_\fn)\right)\right)_\fn\right) = \cR(\epsilon) \in {\rm KS}_r(F/K,\chi).$
\end{proposition}

\begin{proof} The proof of the equality $\Psi\left( \left(s_\fn\left({\rm Rec}_\fn(\epsilon_\fn)\right)\right)_\fn\right) = \cR(\epsilon)$ is the same as that of \cite[Lem. 4.29]{sbA}. The containment $\cR(\epsilon) \in {\rm KS}_r(F/K,\chi)$ follows from \cite[Prop. 4.3]{sbA}.\end{proof}

\subsection{Interpretation of Conjecture \ref{mrs} via Kolyvagin systems}

\begin{theorem}\label{thm2} Conjecture \ref{mrs} is valid for every ideal $\fn$ in $\cN$ if and only if one has
$$\kappa=\cR(\epsilon).$$
\end{theorem}

\begin{proof}
We only show the `if part', since the `only if part' is straightforward from the construction (and unnecessary for the proof of Theorem \ref{main}.)

Suppose that we have $\kappa =\cR(\epsilon)$. Recall that $\kappa$ is constructed by
$$\kappa=\Psi((s_\fn(  \theta_\fn))_\fn).$$
Also, by Proposition \ref{prop reg}, we have
$$\cR(\epsilon)=\Psi((s_\fn({\rm Rec}_\fn(\epsilon_\fn)))_\fn).$$
Since $\Psi$ is injective by Lemma \ref{lempsi}, we have
\begin{align*}
s_\fn( \theta_\fn)=s_\fn({\rm Rec_\fn}(\epsilon_\fn)) \text{ in }{\bigcap}_{\cO[\Gamma]}^r H^1(\cO_{F,\Sigma_\fn},T) \otimes_\ZZ G_\fn.
\end{align*}
If we regard $s_\fn(\theta_\fn)$ as an element of ${\bigcap}_{\cO[\cG_\fn]}^r H^1(\cO_{F(\fn),\Sigma_\fn},T)\otimes_\ZZ G_\fn$, this means
\begin{align}\label{srel}
s_\fn( \theta_\fn)=s_\fn(\iota_\fn({\rm Rec_\fn}(\epsilon_\fn))) \text{ in }{\bigcap}_{\cO[\cG_\fn]}^r H^1(\cO_{F(\fn),\Sigma_\fn},T) \otimes_\ZZ G_\fn,
\end{align}
where $\iota_\fn$ is the injection (\ref{iota}) (extended on $H^1$). In the following, we regard $G_\fn \subset Q_\fn$ via the injection in Lemma \ref{lem1}(iii).
By Lemma \ref{lem0}(ii), we have
\begin{align}\label{trel}
  \theta_\fn=s_\fn(  \theta_\fn)-\sum_{\fd \mid \fn, \ \fd \neq \fn}(-1)^{\nu(\fn/\fd)}  \theta_\fd \prod_{\fq \mid \fn/\fd}P_\fq^\fd
 \text{ in }{\bigcap}_{\cO[\cG_\fn]}^r H^1(\cO_{F(\fn),\Sigma_\fn},T) \otimes_\ZZ Q_\fn.
\end{align}
On the other hand, noting that
$$\varphi_\fq^\fd(-)= v_\fq(-)\cdot P_\fq^\fd$$
holds for $\fq \nmid \fd$ (see \cite[Chap. XIII, Prop. 13]{LF}), we easily deduce
$$\pi_\fd\left( {\rm Rec}_\fn(\epsilon_\fn)\right)={\rm Rec}_\fd(\epsilon_\fd) \prod_{\fq \mid \fn/\fd}P_\fq^\fd$$
and so
\begin{align}\label{recrel}
{\rm Rec}_\fn(\epsilon_\fn)=s_\fn({\rm Rec}_\fn(\epsilon_\fn)) - \sum_{\fd\mid \fn, \ \fd \neq \fn}(-1)^{\nu(\fn/\fd)} {\rm Rec}_\fd(\epsilon_\fd) \prod_{\fq \mid \fn/\fd}P_\fq^\fd \\
\text{ in }{\bigcap}_{\cO[\Gamma]}^r H^1(\cO_{F,\Sigma_\fn},T)\otimes_\ZZ Q_\fn.\nonumber
\end{align}
By (\ref{srel}), (\ref{trel}) and (\ref{recrel}), we deduce by induction on $\nu(\fn)$ that
$$  \theta_\fn=\iota_\fn({\rm Rec}_\fn(\epsilon_\fn)).$$
This is exactly the formulation of Conjecture \ref{mrs}, and so we have proved the claimed result.
\end{proof}

\subsection{The proof of Theorem \ref{main}}\label{proof main}

By Theorem \ref{thm2}, it suffices to prove that
$$\kappa=\cR(\epsilon) \text{ in }{\rm KS}_r(F/K,\chi).$$
Since we know $\kappa_1=c_1=\epsilon_1=\cR(\epsilon)_1$, it is sufficient to prove that the map
$${\rm KS}_r(F/K,\chi) \to  \varprojlim_m {\bigcap}_{\cO_m[\Gamma]}^r \cS(1)_m={\bigcap}_{\cO[\Gamma]}^r U_F; \ x \mapsto x_1$$
is injective.

By Theorem \ref{th main}(ii), we know that the $\cO[\Gamma]$ module ${\rm KS}_r(F/K,\chi)$ is free of rank one.
Let $x \in {\rm KS}_r(F/K,\chi)$ be a basis. It is sufficient to show that
$${\rm Ann}_{\cO[\Gamma]}(x_1)=0.$$
However, since ${\rm Ann}_{\cO[\Gamma]}(x_1)$ is equal to ${\rm Ann}_{\cO[\Gamma]}(\im x_1)$, this follows from the equality
$$
\im (x_1) ={\rm Fitt}_{\cO[\Gamma]}^0({\rm Cl}(E)_\chi)
$$
in 
Theorem~\ref{th main}(ii) and the fact that the ideal class group ${\rm Cl}(E)$ is finite.

This completes the proof of Theorem \ref{main}.

\appendix

\section{The derivative construction in the rank one case}

We fix data $K , p, \chi,L, \Delta, \cO, F,\Gamma, S,\cK$ as in \S \ref{sec eks}. Then the primary aim of this appendix is to give, under Hypothesis \ref{hyp}, an explicit construction for each natural number $m$ of a canonical `derivative' homomorphism
$${\rm ES}_1(\cK/K,S,\chi) \to {\rm KS}_1(F/K,\chi,\cP_m)_m.$$

The construction is described in Theorem \ref{th koly} below and, although it is in principle well-known (via the arguments of Rubin in \cite{R} and of Mazur and Rubin in \cite{MRkoly}), we shall give a fully self-contained construction in the context of this article. This approach is perhaps of some interest itself but also allows us to demonstrate that the maps constructed are Galois equivariant.

The key fact that we must prove is the `congruence relation' described in Theorem \ref{th cong}. Given this result, the strategy of the remainder of the argument is essentially identical to that used in \cite{R} and \cite{MRkoly}, although at various places we have to verify arguments given in loc. cit. for abelian extensions of $\QQ$ extend to the setting of abelian extensions of $K$.

If $\cK$ contains a non-trivial $\ZZ_p$-power extension of $K$, then the argument required to prove Theorem \ref{th cong} is relatively straightforward and works in the setting of general $p$-adic representations. However, in the case that $\cK$ does not contain a $\ZZ_p$-power extension, which is of particular importance in our setting (see Remark \ref{ntz}), details of this sort of argument have not yet appeared in the literature, with only a sketch of this case being discussed by Rubin in \cite[\S 9.1]{R}. 

\subsection{Notation and hypothesis}

Let $H_K$ be the Hilbert class field of $K$ and $K(1)$ the maximal $p$-extension of $K$ inside $H_K$.

First, note that we may assume that $F$ contains $K(1)$, since for any subfield $F'$ of $F/K$ the derivative map ${\rm ES}_1(\cK/K,S,\chi) \to {\rm KS}_1(F'/K,\chi,\cP_m)_m$ is obtained by the composition ${\rm ES}_1(\cK/K,S,\chi) \to {\rm KS}_1(F/K,\chi,\cP_m)_m \to {\rm KS}_1(F'/K,\chi,\cP_m)_m$, where the second map is induced by the norm map for $F/F'$. (Although the set $\cP_m$ defined for $F$ is slightly different from that for $F'$, this does not matter in practice. In fact, one may shrink $\cP_m$ so that the Chebotarev density argument works.)



We now fix a positive integer $m$ and set $M:=p^m$.
We recall some basic notations:
\begin{itemize}
\item $\cP=\cP_m:=\{ \fq \notin S\cup S_p(K) \mid \text{$\fq$ splits completely in $LF H_K(\mu_M,(\cO_K^\times)^{1/M})$}\}$;
\item $\cQ:=\{ \fq \notin S\cup S_p(K) \mid \text{$\fq$ splits completely in $K(\mu_M,(\cO_K^\times)^{1/M})$}\}$;
\item $\cN=\cN_m:=\{\text{square-free products of primes in $\cP$}\}$.
\end{itemize}
For $\fq \in \cQ$, we denote by $K(\fq)$ the unique subfield of the ray class field of $K$ modulo $\fq$ such that $[K(\fq):K(1)] = M$. (Note that this $K(\fq)$ is different from that in \S \ref{sec koly}.) We set $G_\fq:=\Gal(K(\fq)/K(1)) $.
For a square-free product $\fn $ of primes in $\cQ$, we use the following notations:
\begin{itemize}
\item $\nu(\fn):=\# \{\fq \mid \fn\}$;
\item $K(\fn):=\prod_{\fq \mid \fn} K(\fq)$ (compositum);
\item $F(\fn):=F K(\fn)$;
\item $\cG_\fn:=\Gal(F(\fn)/K)\simeq \Gal(LF(\fn)/L)$;
\item $\cH_\fn:=\Gal(F(\fn)/F)\simeq \Gal(LF(\fn)/LF)$ (so $\Gamma=\Gal(F/K)\simeq \cG_\fn /\cH_\fn$);
\item $G_\fn:=\bigotimes_{\fq \mid \fn} G_\fq$.
\end{itemize}
Note that $\cH_\fn\simeq \prod_{\fq \mid \fn}G_\fq$. From this, if $\fm\mid \fn$, then we can regard $\cH_\fm$ both as a quotient and a subgroup of $\cH_\fn$.

We assume that $\cK$ contains $F$ and $K(\fq)$ for all but finitely many primes $\fq$ of $K$.
Suppose that a strict $p$-adic Euler system of rank $1$
$$
c=(c_{F'})_{F'} \in \prod_{F' \in \Omega(\cK/K)}U_{F'}
$$
for $(\cK/K,S,\chi)$ is given (see Definition \ref{p integral}).
We set
$$
c_\fn:=c_{F(\fn)}.
$$

{\it Throughout this appendix, we assume that $U_{F'}:=(\cO_{LF'}^\times)_\chi$ is free as an $\cO$-module for every $F' \in \Omega(\cK/K)$. We also assume Hypothesis \ref{hyp}.}

For each prime $\fq$ of $K$, we fix a place $\fQ$ of $\overline \QQ$ lying above $\fq$. For any finite extension $J/K$, the place of $J$ lying under $\fQ$ is also denoted by $\fQ$.

To simplify the notation, we set
$$E:=LF\text{ and }E(\fn):=LF(\fn).$$
Since any $\fq$ in $\cP$ splits completely in $E$, we often identify the fields $E_\fQ, F_\fQ$ and $K_\fq.$

For any non-archimedean local field $D$, we denote by $\FF_D$ its residue field. The cardinality of $\FF_{K_\fq}$ is denoted by ${\N}\fq$.

For any abelian group $X$, we set $X/M:=X/MX.$ We also often use additive notations: for example, if $J$ is a field and $a \in J^\times$, then we shall denote the $M$-th power of $a$ by $Ma$.

\subsection{The congruence relation}\label{sec cong}

The aim of this subsection is to give a proof of the `congruence relation' of Euler systems (see Theorem \ref{th cong} below).

For any $\fn \in \cN$ and $\fq \in \cP$, let
$$u_\fq: U_{F(\fn)}=(\cO_{E(\fn)}^\times)_\chi \to (\ZZ[\Delta]\otimes_\ZZ \FF_{E(\fn)_\fQ}^\times )_\chi=\cO\otimes_\ZZ \FF_{F(\fn)_\fQ}^\times$$
be the map induced by
$$\cO_{E(\fn)}^\times \to \ZZ[\Delta]\otimes_\ZZ \FF_{E(\fn)_\fQ}^\times ; \ a \mapsto \sum_{\sigma \in \Delta}\sigma^{-1}\otimes(\sigma a \bmod \fQ).$$
(Thus $u_\fq$ is essentially the `modulo primes above $\fq$' map.)

\begin{theorem}\label{th cong}
For any $\fn \in \cN$, $\fq \mid \fn$ and $\sigma \in \cG_\fn$, we have
$$u_\fq(\sigma c_\fn)=\frac{{\N}\fq-1}{M} {\rm Fr}_\fq^{-1} u_\fq(\sigma c_{\fn/\fq}).$$
(Note that $M$ divides ${\N}\fq-1$, since $\fq$ splits completely in $K(\mu_M)$.)
\end{theorem}

\begin{proof}

Choosing a lift $\widetilde \sigma \in G_K$ of $\sigma \in \cG_\fn$ and replacing $(\widetilde \sigma c_{F'})_{F'}$ by $(c_{F'})_{F'}$, we may assume $\sigma=1$.


We first consider the case Hypothesis \ref{hyp}(i). In this case the proof is easier. We follow the argument by Kato in \cite[\S 1]{katoeuler}.

Take a $\ZZ_p$-extension $K_\infty/K$ satisfying the condition in Hypothesis \ref{hyp}(i).
We denote the $m$-th rayer of $K_\infty/K$ by $K_m$. We set
\[ F(\fn)_m:=F(\fn)K_m\,\,\text{ and }\,\, c_{\fn,m}:=c_{F(\fn)_{m}}.\]
We can define $U_{F(\fn)_m} \to \cO\otimes_\ZZ \FF_{F(\fn)_{m,\fQ}}^\times$ in the same way as $u_\fq$, and denote it also by $u_\fq$.
Since $S_{\rm ram}(K_\infty/K) \subset S$, we see by the norm relation that $(c_{\fn,m})_m$ lies in the inverse limit $\varprojlim_m U_{F(\fn)_m}$. It is sufficient to prove the equality
$$(u_\fq(c_{\fn,m}))_m=\left( \frac{{\N}\fq-1}{M}{\rm Fr}_\fq^{-1} u_\fq(c_{\fn/\fq,m})\right)_m$$
in $\varprojlim_m (\cO \otimes_\ZZ \FF_{F(\fn)_{m,\fQ}}^\times)=\varprojlim_m (\cO \otimes_\ZZ \FF_{F(\fn/\fq)_{m,\fQ}}^\times)$. Since $\fq$ does not split completely in $K_\infty$, the module $\varprojlim_m (\cO \otimes_\ZZ \FF_{F(\fn)_{m,\fQ}}^\times)$ is $\cO$-free, so it is sufficient to prove
\begin{eqnarray}\label{cong eq}
M \cdot u_\fq(c_{\fn,m})=({\N}\fq-1){\rm Fr}_\fq^{-1}u_\fq(c_{\fn/\fq,m})
\end{eqnarray}
for every $m$. By the norm relation, we have
$${\N}_{G_\fq}(c_{\fn,m})=(1-{\rm Fr}_\fq^{-1})c_{\fn/\fq,m},$$
where ${\N}_{G_\fq}:=\sum_{\sigma \in G_\fq}\sigma$. (We identify $G_\fq=\Gal(F(\fn)_m/F(\fn/\fq)_m)$.) Since $G_\fq$ acts trivially on $\FF_{F(\fn)_{m,\fQ}}^\times=\FF_{F(\fn/\fq)_{m,\fQ}}^\times$ and $M=\# G_\fq$, we have
$$M \cdot u_\fq(c_{\fn,m})=(1-{\rm Fr}_\fq^{-1})u_\fq(c_{\fn/\fq,m}).$$
Since $1-{\rm Fr}_\fq^{-1}=({\N}\fq-1){\rm Fr}_\fq^{-1}+(1-{\N}\fq {\rm Fr}_\fq^{-1})$ and $1-{\N}\fq {\rm Fr}_\fq^{-1}$ annihilates $\FF_{F(\fn/\fq)_{m,\fQ}}^\times$, we obtain (\ref{cong eq}). Thus the theorem is proved in the case (a).

Next, we consider the case Hypothesis \ref{hyp}(ii). Our proof is almost the same as that of \cite[Th. 2.1]{rubincrelle} by Rubin, but, since it is assumed in loc. cit. that $\fq$ splits completely in $E(\fn/\fq)$, we slightly modify the argument. 

We set $K_M:=K(\mu_M,(\cO_K^\times)^{1/M})$. The assumption $L \not\subset K(\mu_p)$ ensures the existence of $\gamma \in \Gal(E(\fn/\fq)/E\cap K_M)$ such that $\gamma $ acts non-trivially on $L$ and trivially on $F(\fn/\fq)$. Since the order of $\chi$ is prime to $p$, we see that $1-\chi(\gamma) \in \cO^\times$.
So it is sufficient to show
\begin{eqnarray}\label{gamma cong}
(1-\gamma)u_\fq( c_\fn)=(1-\gamma)\frac{{\N}\fq-1}{M} {\rm Fr}_\fq^{-1} u_\fq( c_{\fn/\fq}).
\end{eqnarray}
By Lemma \ref{cheb} below, we can choose a prime $\fr \in \cQ$ with $\fr \nmid\fn$ such that
\begin{itemize}
\item ${\rm Fr}_\fr=\gamma^{-1}$ on $E(\fn/\fq)$,
\item there exists $m \in \ZZ$ such that ${\N}_{G_\fr}=mM$ on $\FF_{E(\fn\fr)_\fQ}^\times=\FF_{F(\fn\fr/\fq)_\fQ}^\times$.
\end{itemize}
(The proof of Lemma \ref{cheb} actually shows that $G_\fr$ is contained in the decomposition group in $\Gal(F(\fn\fr/\fq)/F)$ at $\fq$, so it acts on $\FF_{F(\fn\fr/\fq)_\fQ}^\times$.)
By the norm relation, we have equalities
\begin{eqnarray}\label{eq1}\hskip 0.3truein
M \cdot u_\fq(c_{\fn\fr})={\N}_{G_\fq}u_\fq(c_{\fn\fr})=(1-{\rm Fr}_\fq^{-1})u_\fq(c_{\fn\fr/\fq})=({\N}\fq-1){\rm Fr}_\fq^{-1} u_\fq(c_{\fn\fr/\fq}),
\end{eqnarray}
\begin{eqnarray}\label{eq2}
mM\cdot u_\fq(c_{\fn\fr})={\N}_{G_\fr}u_\fq(c_{\fn\fr})=(1-{\rm Fr}_\fr^{-1})u_\fq(c_{\fn})=(1-\gamma)u_\fq(c_\fn),
\end{eqnarray}
\begin{eqnarray}\label{eq3}\hskip 0.2truein
mM \cdot u_\fq(c_{\fn\fr/\fq})={\N}_{G_\fr}u_\fq(c_{\fn\fr/\fq}) =(1-{\rm Fr}_\fr^{-1})u_\fq(c_{\fn/\fq})=(1-\gamma) u_\fq(c_{\fn/\fq}).
\end{eqnarray}
Using these equalities, we compute
\begin{align*}
(1-\gamma)u_\fq(c_\fn) &\stackrel{(\ref{eq2})}{=} mM \cdot u_\fq(c_{\fn\fr})
\\
&\stackrel{(\ref{eq1})}{=}m({\N}\fq -1){\rm Fr}_\fq^{-1} u_\fq(c_{\fn\fr/\fq})
\stackrel{(\ref{eq3})}{=}(1-\gamma)\frac{{\N}\fq-1}{M}{\rm Fr}_\fq^{-1} u_\fq(c_{\fn/\fq}).
\end{align*}
This is the desired equality (\ref{gamma cong}).
\end{proof}

\begin{lemma}\label{cheb}
Let $\fn \in \cN$ and $\fq \mid \fn$. For any $\gamma \in \Gal(E(\fn/\fq)/E\cap K_M)$, there exists a prime $\fr \in \cQ$ with $\fr \nmid\fn$ such that
\begin{itemize}
\item[(i)] ${\rm Fr}_\fr=\gamma$ on $E(\fn/\fq)$,
\item[(ii)] there exists $m \in \ZZ$ such that ${\N}_{G_\fr}=mM$ on $\FF_{E(\fn\fr)_\fQ}^\times=\FF_{F(\fn\fr/\fq)_\fQ}^\times$.
\end{itemize}
\end{lemma}

\begin{proof}
We follow the argument by Rubin in \cite[Lem. 2.2]{rubincrelle}.
Since $\fq$ splits completely in $H_K$, it is a principal ideal: there exists a prime element $q \in \cO_K$ such that $\fq=(q)$. Let $\sigma$ be a generator of $\Gal(K_M(q^{1/p})/K_M)$. Considering ramification, one sees $E(\fn/\fq)\cap K_M(q^{1/p})=E\cap K_M$. By the Chebotarev density theorem, there exists a prime $\fr$ with $\fr \nmid\fn$ such that
\begin{itemize}
\item ${\rm Fr}_\fr=\gamma$ in $\Gal(E(\fn/\fq)/E\cap K_M)$,
\item ${\rm Fr}_\fr=\sigma$ in $\Gal(K_M(q^{1/p})/E\cap K_M)$.
\end{itemize}
Then the condition (i) is satisfied. Since ${\rm Fr}_\fr=\sigma=1$ on $K_M$, we see that $\fr \in \cQ$. We need to check the condition (ii).

Let $\cD_\fq \subset \Gal(F(\fn\fr/\fq)/F)$ be the decomposition group at the (fixed) prime above $\fq$. We claim
$$\Gal(F(\fn\fr/\fq)/F(\fn/\fq))(\simeq G_\fr) \subset \cD_\fq.$$
To prove this claim, it is sufficient to show that $G_\fr=\Gal(K(\fr)/K(1))$ is generated by ${\rm Fr}_\fq$. By class field theory, we have an isomorphism
$$\FF_{K_\fr}^\times/M \xrightarrow{\sim} G_\fr,$$
which sends $q^{-1}$ to ${\rm Fr}_\fq$. Since ${\rm Fr}_\fr$ generates $\Gal(K_M(q^{1/p})/K_M)$ by the choice of $\fr$, we see that $q \notin  (\FF_{K_\fr}^\times)^p$. Hence $q^{-1}$ generates $\FF_{K_\fr}^\times/M$, and so ${\rm Fr}_\fq$ generates $G_\fr$. We have proved the claim.

By the above claim, we see that a generator of $\Gal(F(\fn\fr/\fq)/F(\fn/\fq))$ is of the form ${\rm Fr}_\fq^a$ with some $a \in \ZZ$, so ${\N}_{G_\fr}$ acts on $\FF_{F(\fn\fr/\fq)_\fQ}^\times$ via $\sum_{i=0}^{M-1}{\rm Fr}_\fq^{ai}$.
Since ${\rm Fr}_\fq={\N}\fq$ on $\FF_{F(\fn\fr/\fq)_\fQ}^\times$, it is sufficient to show that $\sum_{i=0}^{M-1} {\N}\fq^{ai}$ is divisible by $M$. But this follows from
${\N}\fq \equiv 1 \pmod M$.
\end{proof}

\subsection{Kolyvagin derivatives}

The aim of this subsection is to review the construction of `Kolyvagin derivatives' (in Definition \ref{def der}) and prove Theorem \ref{th fs}, the so-called `finite-singular relation'. The argument used is essentially the same as that in \cite[Prop. 2.4]{rubinlang}.
 Note that in our $\GG_m$ case one of the key ingredients of the proof is Hilbert's Theorem 90 (more precisely, the existence of the element $\beta_\fn$ in Lemma \ref{lem a1} below), whose analogue does not seem to hold in the case of a general motive. However, as observed by several authors (such as Nekov\'a\v r \cite{nekovar}, Perrin-Riou \cite{PR}, Kato \cite{katoeuler} and Rubin \cite{R}), one can adapt the argument below in order to prove a generalization of Theorem \ref{th fs}. 

For each $\fq \in \cP$, we fix a generator $\sigma_\fq$ of $G_\fq$. Then Kolyvagin's derivative operator is defined by
$$D_\fq:=\sum_{i=1}^{M-1}i\sigma_\fq^{i} \in \ZZ[G_\fq].$$
By computation, one checks the so-called `telescoping identity'
\begin{eqnarray}\label{tele}
(\sigma_\fq-1)D_\fq=M-{\N}_{G_\fq},
\end{eqnarray}
where ${\N}_{G_\fq}:=\sum_{i=0}^{M-1}\sigma_\fq^i$.
For $\fn \in \cN$, we identify $\cH_\fn$ with $\prod_{\fq \mid \fn}G_\fq$ and then set
$$D_\fn:=\prod_{\fq \mid \fn}D_\fq \in \ZZ[\cH_\fn].$$

\begin{lemma}\label{lem a1}\
\begin{itemize}
\item[(i)] For any $\fn \in \cN$ and $\sigma \in \cH_\fn$ we have $(\sigma-1)D_\fn c_\fn \in M \cdot U_{F(\fn)}.$
\item[(ii)]For any $\fn \in \cN$, there exists an element $\beta_\fn$ of $(E(\fn)^\times)_\chi$ with the following property: for any $\sigma \in \cH_\fn$ we have
$$
(\sigma-1)\beta_\fn=\frac 1M (\sigma-1)D_\fn c_\fn.
$$
(Note that `$\frac{1}{M}$' is well-defined, since $U_{F(\fn)}$ is assumed to be $\cO$-free.)
\end{itemize}
\end{lemma}

\begin{proof}
Claim (i) is easily checked by computation using (\ref{tele}). See \cite[Lem. 6.10]{bss} for example. (See also \cite[Lem. 2.1]{rubinlang} and \cite[Lem. 4.4.2(i)]{R}.)

We show claim (ii). Note that
$$\cH_\fn(=\Gal(E(\fn)/E)) \to (E(\fn)^\times)_\chi; \ \sigma \mapsto \frac 1M (\sigma-1)D_\fn c_\fn$$
is a 1-cocycle. So it is sufficient to show that $H^1(\cH_\fn, (E(\fn)^\times)_\chi)=0$. This is deduced from the vanishing of $H^1(\cH_\fn,E(\fn)^\times)$ that is a consequence of Hilbert's Theorem 90.\end{proof}

\begin{definition}\label{def der}
Let $\fn \in \cN$. We define the Kolyvagin derivative by
$$\kappa_\fn'=\kappa'(c_\fn):=D_\fn c_\fn -M \beta_\fn.$$
One sees that this element lies in $(E^\times)_\chi$, and its image in $(E^\times/M)_\chi$ does not depend on the choice of $\beta_\fn$. We often regard $\kappa_\fn'$ as an element of $(E^\times/M)_\chi$.
\end{definition}

We recall some notations from \S \ref{sec koly}. For any prime $\fq$ of $K$, let
$$v_\fq : (E^\times/M)_\chi \to (\ZZ/M[\Delta \times \Gamma])_\chi =\cO/M[\Gamma]$$
be the map induced by (\ref{def v}).
For $\fq \in \cP$, let
$$\varphi_\fq: (E^\times/M)_\chi \to \cO/M[\Gamma] \otimes_\ZZ G_\fq$$
be the map induced by (\ref{def phi}).
We set
$$\cS^\fn:=\{a \in (E^\times/M)_\chi \mid v_\fq(a)=0 \text{ for every $\fq \nmid \fn $}\}.$$
(This is $\cS^\fn_m$ in \S \ref{sec koly}.)

\begin{theorem}\label{th fs}
Let $\fn \in \cN$.
\begin{itemize}
\item[(i)] $\kappa'_\fn \in \cS^\fn$.
\item[(ii)] For any $\fq \mid \fn$, one has $v_\fq(\kappa_\fn')\otimes \sigma_\fq=\varphi_\fq (\kappa'_{\fn/\fq})$ in $\cO/M[\Gamma]\otimes_\ZZ G_\fq.$
\end{itemize}
\end{theorem}

\begin{proof}
Claim (i) follows by noting that $\kappa_\fn'=D_\fn c_\fn$ in $(E(\fn)^\times/M)_\chi$, $c_\fn \in U_{F(\fn)}(=(\cO_{E(\fn)}^\times)_\chi)$ and $E(\fn)/E$ is unramified outside $\fn$.

We prove claim (ii). We take $\tau \in \Gamma$. We fix a lift of $\tau$ in $\cG_\fn(=\Gal(F(\fn)/K))$ and denote it also by $\tau$. Recall that $\kappa_\fn'=D_\fn c_\fn-M\beta_\fn \in (E(\fn)^\times)_\chi$ by definition. Since $E(\fq)/E$ is totally ramified at primes above $\fq$, we have
\begin{eqnarray}\label{eq ord}
\sum_{\sigma \in \Delta} {\rm ord}_\fQ(\sigma\tau \kappa_\fn')\sigma^{-1}&=&\sum_{\sigma \in \Delta}{\rm ord}_\fQ(\sigma\tau (D_\fn c_\fn -M\beta_\fn))\sigma^{-1}\\
&=&-\sum_{\sigma \in \Delta}{\rm ord}_{E(\fn)_\fQ}(\sigma\tau \beta_\fn)\sigma^{-1} \nonumber
\end{eqnarray}
in $(\ZZ/M[\Delta])_\chi=\cO/M $, where $\ord_\fQ: E^\times \to \ZZ$ and $\ord_{E(\fn)_\fQ}: E(\fn)^\times \to \ZZ$ are the normalized valuations.

Let
$$u_\fq : \{a \in (E(\fn)^\times)_\chi \mid \text{$a$ is a unit at all primes above $\fq$}\} \to \cO \otimes_\ZZ \FF_{E(\fn)_\fQ}^\times $$
be the map induced by $u_\fq $ in \S \ref{sec cong}.
This map also induces
$$ u_\fq : \cS^{\fn/\fq} \to \cO\otimes_\ZZ \FF_{E_{\fQ}}^\times/M=\cO\otimes_\ZZ \FF_{K_\fq}^\times/M .$$

By (\ref{eq ord}) and Lemma \ref{lem fs} below, it is sufficient to prove
\begin{eqnarray}\label{fs suff}
u_\fq((1-\sigma_\fq)\tau \beta_\fn)=\frac{1-{\N}\fq}{M}u_\fq(\tau \kappa_{\fn/\fq}') \text{ in }\cO\otimes_\ZZ \FF_{E(\fn)_\fQ}^\times,
\end{eqnarray}
where $\frac{1-{\N}\fq}{M}$ denotes the map $\cO\otimes_\ZZ \FF_{K_{\fq}}^\times/M \to \cO\otimes_\ZZ \FF_{E(\fn)_\fQ}^\times $ induced by
$$\FF_{K_\fq}^\times/M \stackrel{\frac{1-{\N}\fq}{M}}{\hookrightarrow} \FF_{K_\fq}^\times \subset \FF_{E(\fn)_\fQ}^\times.$$
We compute
\begin{eqnarray*}
&&u_\fq((1-\sigma_\fq)\tau\beta_\fn ) \\
&=& u_\fq\left( \frac 1M (1-\sigma_\fq)\tau D_\fn c_\fn\right) \text{ (by Lemma \ref{lem a1}(ii))}\\
&=& u_\fq\left( \frac 1M ({\N}_{G_\fq}-M)\tau D_{\fn/\fq}c_\fn \right) \text{ (by (\ref{tele}))}\\
&=&u_\fq\left(\frac 1M (1-{\rm Fr}_\fq^{-1})\tau D_{\fn/\fq}c_{\fn/\fq} -\tau D_{\fn/\fq}c_\fn\right) \text{ (by the norm relation)}\\
&=&u_\fq\left(\frac 1M (1-{\rm Fr}_\fq^{-1})\tau D_{\fn/\fq}c_{\fn/\fq} -\frac{{\N}\fq-1}{M}{\rm Fr}_\fq^{-1}\tau D_{\fn/\fq}c_{\fn/\fq}\right) \text{ (by Theorem \ref{th cong})}\\
&=&u_\fq\left((1-{\rm Fr}_\fq^{-1})\tau\beta_{\fn/\fq}-\frac{{\N}\fq-1}{M}{\rm Fr}_\fq^{-1}\tau D_{\fn/\fq}c_{\fn/\fq}\right) \text{ (by Lemma \ref{lem a1}(ii))}\\
&=&u_\fq\left(({\N}\fq-1){\rm Fr}_\fq^{-1}\tau\beta_{\fn/\fq}-\frac{{\N}\fq-1}{M}{\rm Fr}_\fq^{-1}\tau D_{\fn/\fq}c_{\fn/\fq}\right) \text{ (by $(1-{\N}\fq {\rm Fr}_\fq^{-1})u_\fq(\beta_{\fn/\fq})=0$)}\\
&=&\frac{1-{\N}\fq}{M}{\rm Fr}_\fq^{-1} u_\fq(\tau (D_{\fn/\fq}c_{\fn/\fq}-M \beta_{\fn/\fq})) \\
&=& \frac{1-{\N}\fq}{M}u_\fq(\tau \kappa_{\fn/\fq}') \text{ (by ${\rm Fr}_\fq=1$ on $\FF_{K_\fq}$).}
\end{eqnarray*}
Thus (\ref{fs suff}) is proved.
\end{proof}

\begin{lemma}\label{lem fs}
Let $\kappa \in \cS^{\fn/\fq}$. Suppose that there exists $\beta \in (E(\fn)^\times)_\chi$ such that
$$u_\fq((1-\sigma_\fq  )\beta) =\frac{1-{\N}\fq}{M} u_\fq(\kappa) \text{ in }\cO\otimes_\ZZ \FF_{E(\fn)_\fQ}^\times . $$
Then we have
$$\sum_{\sigma \in \Delta}\sigma^{-1}\otimes{\rm rec}_\fQ(\sigma \kappa)=-\left(\sum_{\sigma \in \Delta}{\rm ord}_{E(\fn)_\fQ}(\sigma \beta)\sigma^{-1}\right) \otimes \sigma_\fq$$
in $(\ZZ/M[\Delta]\otimes_\ZZ G_\fq)_\chi =\cO/M \otimes_\ZZ G_\fq,$ where ${\rm rec}_\fQ: E^\times \to E_\fQ^\times \to \Gal(E(\fq)_\fQ/E_\fQ)=G_\fq$ is the local reciprocity map at $\fQ$ and $\ord_{E(\fn)_\fQ}: E(\fn)^\times \to \ZZ$ is the normalized valuation.\end{lemma}

\begin{proof} This is a well-known fact from algebraic number theory (see \cite[Lem. 2.3]{rubinlang} or \cite[Lem. 3.2]{kurihara}).
\end{proof}

\subsection{Kolyvagin systems}

The aim of this subsection is to modify the system $(\kappa_\fn')_\fn$ to construct a Kolyvagin system (see Theorem \ref{th koly}). Our argument is parallel to that of Mazur and Rubin in \cite[App. A]{MRkoly}, although in loc. cit. $K$ is assumed to be $\QQ$. 

For $\fq \in \cP$, recall from \S \ref{sec koly} that
we have a canonical decomposition
$$K_\fq^\times/M=\Pi_\fq \times \FF_{K_\fq}^\times/M .$$
Let
$$\widetilde l_\fq : (E^\times/M)_\chi \to (\ZZ[\Delta\times \Gamma] \otimes_\ZZ K_\fq^\times/M)_\chi=\cO[\Gamma] \otimes_\ZZ K_\fq^\times/M$$
be the map induced by (\ref{def l}).
We define
$$\cS(\fn):=\{a \in \cS^\fn \mid \widetilde l_\fq(a) \in \cO[\Gamma] \otimes_\ZZ \Pi_\fq \text{ for every $\fq \mid \fn$}\}.$$
(This is $\cS(\fn)_m$ in \S \ref{sec koly}.)

We recall the definition of Kolyvagin systems in the rank one case. (See Definition \ref{def koly}.)

\begin{definition}
A Kolyvagin system (of rank one for $(F/K,\chi,m,\cP_m)$) is a collection
$$(x_\fn)_\fn \in \prod_{\fn \in \cN} \cS(\fn)\otimes_\ZZ G_\fn$$
which satisfies the following property: for any $\fn \in \cN$ and $\fq \mid \fn$, we have
$$v_\fq(x_\fn)=\varphi_\fq(x_{\fn/\fq}).$$
The module of Kolyvagin systems is denoted by ${\rm KS}_1(F/K,\chi,\cP_m)_m$.
\end{definition}

Let $\mathfrak{S}(\fn)$ be the set of permutations of $\{\fq \mid \fn\}$. For $\pi \in \mathfrak{S}(\fn)$, define
$$\fd_\pi:=\prod_{\fq \mid \fn, \ \pi(\fq)=\fq}\fq \in \cN. $$
For $\fq ,\fr \in \cP$, we denote by ${\rm Fr}_\fr^\fq$ the Frobenius element of $\fr$ regarded as an element of $G_\fq$. (When $\fq=\fr$, define ${\rm Fr}_\fq^\fq:=1$.)

\begin{theorem}\label{th koly}
For $\fn \in \cN$, we set
$$\kappa_\fn=\kappa(c)_\fn:=\sum_{\pi \in \mathfrak{S}(\fn)} {\rm sgn}(\pi) \kappa_{\fd_\pi}' \otimes \bigotimes_{\fq \mid \fd_\pi}\sigma_\fq  \otimes \bigotimes_{\fq \mid \fn/\fd_\pi}{\rm Fr}_{\pi(\fq)}^\fq \in \cS^\fn\otimes_\ZZ G_\fn.$$
Then $\kappa_\fn$ belongs to $\cS(\fn) \otimes_\ZZ G_\fn$ and the collection $\kappa=(\kappa_\fn)_{\fn \in \cN}$ is a Kolyvagin system. In particular, the assignment $c \mapsto \kappa$ gives a canonical homomorphism
$${\rm ES}_1(F/K,S,\chi) \to {\rm KS}_1(F/K,\chi,\cP_m)_m.$$
\end{theorem}

By using Theorem \ref{th fs}, one easily verifies that $\kappa$ satisfies
$$v_\fq(\kappa_\fn)=\varphi_\fq(\kappa_{\fn/\fq})$$
for any $\fn \in \cN$ and $\fq \mid \fn$. So it is enough to prove that $\kappa_\fn$ belongs to $\cS(\fn) \otimes_\ZZ G_\fn. $

We define a map $\widetilde l_{\fq,f}: (E^\times/M)_\chi \to \cO[\Gamma]\otimes_\ZZ \FF_{K_\fq}^\times/M$ by the composition $$\widetilde l_{\fq,f}: (E^\times/M)_\chi \xrightarrow{\widetilde l_\fq} \cO[\Gamma]\otimes_\ZZ K_\fq^\times/M \twoheadrightarrow \cO[\Gamma] \otimes_\ZZ \FF_{K_\fq}^\times/M,$$
where the second map is induced by the projection $K_\fq^\times/M \twoheadrightarrow \FF_{K_\fq}^\times/M$.

For $\fq ,\fr \in \cP$, define $e_\fr^\fq \in \ZZ/M$ by
$${\rm Fr}_\fr^\fq =\sigma_\fq^{e_{\fr}^\fq} \text{ in }G_\fq. $$

Theorem \ref{th koly} is reduced to the following lemma.

\begin{lemma}[{{\cite[Th. A.4]{MRkoly}}}]\label{compute}
For any $\fn \in \cN$ and a prime $\fr \mid \fn$, we have
$$\widetilde l_{\fr,f}(\kappa'_\fn) =\sum_{\pi\in \mathfrak{S}_1(\fn), \ \pi(\fr)\neq \fr}(-1)^{\nu(\fn/\fd_\pi)}\left( \prod_{\fq \mid \fn/\fd_\pi} e_{\pi(\fq)}^\fq \right)\widetilde l_{\fr,f}(\kappa'_{\fd_\pi}),$$
where $\mathfrak{S}_1(\fn) \subset \mathfrak{S}(\fn)$ is the subset of cyclic permutations.

\end{lemma}

We give a proof of Theorem \ref{th koly} by using this lemma.

\begin{proof}[Proof of Theorem \ref{th koly}]
As explained above, it is sufficient to show $\kappa_\fn \in \cS(\fn)\otimes_\ZZ G_\fn$. One sees that this is equivalent to showing that
\begin{eqnarray}\label{zero}
\sum_{\pi \in \mathfrak{S}(\fn)}{\rm sgn}(\pi) \left( \prod_{\fq \mid \fn/\fd_\pi} e_{\pi(\fq)}^\fq \right)\widetilde l_{\fr,f}(\kappa'_{\fd_\pi})=0
\end{eqnarray}
for every prime $\fr \mid \fn$. We compute
\begin{eqnarray*}
&&\sum_{\pi \in \mathfrak{S}(\fn)}{\rm sgn}(\pi) \left( \prod_{\fq \mid \fn/\fd_\pi} e_{\pi(\fq)}^\fq \right) \kappa'_{\fd_\pi} \\
&=& \sum_{\pi \in \mathfrak{S}(\fn), \ \pi(\fr)=\fr}{\rm sgn}(\pi) \left( \prod_{\fq \mid \fn/\fd_\pi} e_{\pi(\fq)}^\fq \right) \kappa'_{\fd_\pi} \\
&&+ \sum_{\pi \in \mathfrak{S}(\fn), \ \pi(\fr)=\fr}\sum_{\pi' \in \mathfrak{S}_1(\fd_\pi),\ \pi'(\fr)\neq \fr}{\rm sgn}(\pi\pi') \left( \prod_{\fq \mid \fn/\fd_{\pi'}} e_{\pi\pi'(\fq)}^\fq \right) \kappa'_{\fd_{\pi'}} \\
&=& \sum_{ \pi \in \mathfrak{S}(\fn), \ \pi(\fr)=\fr}{\rm sgn}(\pi)  \left( \prod_{\fq \mid \fn/\fd_\pi} e_{\pi(\fq)}^\fq \right) s_\pi,
\end{eqnarray*}
where
$$s_\pi:=\kappa_{\fd_\pi}'-\sum_{\pi' \in \mathfrak{S}_1(\fd_\pi), \ \pi'(\fr)\neq \fr}(-1)^{\nu(\fd_\pi/\fd_{\pi'})} \left( \prod_{\fq \mid \fd_\pi/\fd_{\pi'}} e_{\pi'(\fq)}^\fq \right)  \kappa'_{\fd_{\pi'}}.$$
Since $\widetilde l_{\fr,f}(s_\pi)=0$ by Lemma \ref{compute}, we obtain (\ref{zero}).
\end{proof}

The rest of this appendix is devoted to the proof of Lemma \ref{compute}.

We use the following `modified Euler system'
$$
\omega_\fn:=\sum_{\fd \mid \fn}(-1)^{\nu(\fn/\fd)} \left( \prod_{\fq \mid \fn/\fd}\frac{{\N}\fq-1}{M} {\rm Fr}_\fq^{-1}\right)c_\fd \in U_{F(\fn)}.
$$
(Such a modification was considered by Kato \cite[\S 2.2]{katoeuler} and, more generally, by Rubin \cite[\S 9.6]{R}.)
By computation, one checks that the system $(\omega_\fn)_\fn$ satisfies the following norm relation:
\begin{eqnarray}\label{omega rel}
{\N}_{G_\fq}\omega_\fn=(1-{\N}\fq{\rm Fr}_\fq^{-1})\omega_{\fn/\fq} \text{ for any $\fq \mid \fn$}.
\end{eqnarray}
(The norm relation of this form is used in the definition of Euler systems by Mazur and Rubin in \cite[Def. 3.2.2]{MRkoly}.)
Also, by using Theorem \ref{th cong}, one checks that
\begin{eqnarray}\label{cong 2}
u_\fq(\sigma \omega_\fn)=0\text{ for any $\fq \mid \fn$ and $\sigma \in \cG_\fn$}.
\end{eqnarray}
An analogue of Lemma \ref{lem a1} holds for $\omega_\fn$: there exists $\gamma_\fn \in (E(\fn)^\times)_\chi$ such that
\begin{eqnarray*}
(\sigma-1)\gamma_\fn=\frac 1M (\sigma-1)D_\fn \omega_\fn
\end{eqnarray*}
for any $\sigma \in \cH_\fn$. As in Definition \ref{def der}, we define
$$\kappa'(\omega_\fn):=D_\fn \omega_\fn -M \gamma_\fn \in (E^\times)_\chi.$$
One checks that
$$\kappa_\fn'(=\kappa'(c_\fn))=\kappa'(\omega_\fn) \text{ in }(E^\times/M)_\chi.$$
So we may replace $\kappa_\fn'$ by $\kappa'(\omega_\fn)$.

Let
$$l_{\fq,f}^1: (E^\times/M)_\chi \to (\ZZ[\Delta]\otimes_\ZZ \FF_{K_\fq}^\times/M)_\chi=\cO\otimes_\ZZ \FF_{K_\fq}^\times/M$$
be the map induced by
$$E^\times/M \to \ZZ[\Delta]\otimes_\ZZ K_\fq^\times/M \twoheadrightarrow \ZZ[\Delta]\otimes_\ZZ \FF_{K_\fq}^\times/M,$$
where the first map is induced by $a \mapsto \sum_{\sigma \in \Delta}\sigma^{-1}\otimes l_\fq(\sigma a)$ and the second by the projection $K_\fq^\times/M \twoheadrightarrow \FF_{K_\fq}^\times/M$. This map is related with $\widetilde l_{\fq,f}: (E^\times/M)_\chi \to \cO[\Gamma]\otimes_\ZZ \FF_{K_\fq}^\times/M$ by
$$\widetilde l_{\fq,f}=\sum_{\sigma \in \Gamma}\sigma^{-1}\otimes l_{\fq,f}^1(\sigma(-)).$$

In the following, we fix $\fn \in \cN$.

\begin{lemma}[{{\cite[Prop. A.15]{MRkoly}}}]\label{lem xi}
Let $\fm \mid \fn$ and $ \fq \mid \fn$.
\begin{itemize}
\item[(i)] $(1-{\N} \fq{\rm Fr}_ \fq^{-1})D_\fm \omega_\fm \in M \cdot U_{F(\fm)}$.
(When $\fq \mid \fm$, the element ${\rm Fr}_ \fq\in \cH_{\fm/ \fq}$ is regarded as an element of $\cH_\fm$ via $\cH_{\fm/ \fq}\hookrightarrow \cH_{\fm}$.)
\item[(ii)] We set
$$
\xi_{\fm, \fq}:=  \frac 1M (1-{\N} \fq{\rm Fr}_ \fq^{-1})D_\fm \omega_\fm \in U_{F(\fm)}.
$$
Then for any $\sigma \in \cG_\fm$ we have
$$
\frac{1-{\N}\fq}{M} l_{ \fq,f}^1(\sigma \kappa'_\fm)=u_ \fq(\sigma \xi_{\fm, \fq}) \text{ in }\cO \otimes_\ZZ \FF_{E(\fn)_\fQ}^\times,
$$
where $\frac{1-{\N}\fq}{M}$ denotes the injection $\cO\otimes_\ZZ \FF_{K_{\fq}}^\times/M \hookrightarrow \cO\otimes_\ZZ \FF_{E(\fn)_\fQ}^\times $ induced by
$$
\FF_{K_\fq}^\times/M \stackrel{  \frac{1-{\N}\fq}{M}}{\hookrightarrow} \FF_{K_\fq}^\times \subset \FF_{E(\fn)_\fQ}^\times.
$$
\end{itemize}
\end{lemma}

\begin{proof}
The proof of claim (i) is the same as that of Lemma \ref{lem a1}(i).

We prove claim (ii). We may assume $\sigma=1$.
As in Lemma \ref{lem a1}(ii), there exists $\gamma_{\fm,\fq} \in (E(\fn)^\times)_\chi$ such that
\begin{eqnarray}\label{gamma rel}
(\tau-1)\gamma_{\fm,\fq}= \frac 1M (\tau-1)D_\fm \omega_\fm
\end{eqnarray}
for any $\tau \in \cH_{\fn/\fq}= \Gal(E(\fn)/E(\fq))$. We see that
$$D_\fm \omega_\fm - M \gamma_{\fm,\fq} \in (E(\fq)^\times)_\chi$$
and
$$\kappa_\fm' = \kappa'(\omega_\fm)= D_\fm \omega_\fm -M \gamma_{\fm,\fq} \text{ in }(E(\fq)^\times/M)_\chi.$$
Since $E(\fn)/E(\fq)$ is unramified at primes above $\fq$, we may assume that $\gamma_{\fm,\fq}$ is a unit at primes above $\fq$. Since the projection map $K_\fq^\times/M \twoheadrightarrow \FF_{K_\fq}^\times/M \subset E(\fq)_\fQ^\times/M$ coincides with the map induced by the inclusion $K_\fq \hookrightarrow E(\fq)_\fQ$, we see that
$$l_{\fq,f}^1(\kappa_\fm')=u_\fq(D_\fm \omega_\fm - M\gamma_{\fm,\fq}). $$
(Compare \cite[Prop. A.8]{MRkoly}.)
So it is sufficient to prove
$$\frac{1-{\N}\fq}{M}u_\fq(D_\fm \omega_\fm - M \gamma_{\fm,\fq})=u_\fq(\xi_{\fm,\fq}).$$
We compute
\begin{eqnarray*}
&&\frac{1-{\N}\fq}{M}u_\fq(D_\fm \omega_\fm - M \gamma_{\fm,\fq}) \\
&=& \frac{1-{\N}\fq}{M}u_\fq(D_\fm \omega_\fm) - (1-{\N}\fq)u_\fq(\gamma_{\fm,\fq})\\
&=& \frac{1-{\N}\fq}{M} u_\fq(D_\fm \omega_\fm)-({\N}\fq {\rm Fr}_\fq^{-1}-{\N}\fq)u_\fq(\gamma_{\fm,\fq}) \text{ (by $(1-{\N}\fq{\rm Fr}_\fq^{-1})u_\fq(\gamma_{\fm,\fq})=0$)}\\
&=& u_\fq\left( \frac{1-{\N}\fq}{M} D_\fm \omega_\fm-\frac{{\N}\fq}{M} ({\rm Fr}_\fq^{-1}-1)D_\fm \omega_\fm \right) \text{ (by (\ref{gamma rel}))}\\
&=&u_\fq\left( \frac 1M(1-{\N}\fq{\rm Fr}_\fq^{-1}) D_\fm \omega_\fm \right)=u_\fq(\xi_{\fm,\fq}).
\end{eqnarray*}
This proves the claim.
\end{proof}

\begin{lemma}[{{\cite[Lem. A.12]{MRkoly}}}]\label{lem xi 2}
Let $\fm \mid \fn$, $\fq \mid \fm$ and $\fr \mid \fn$. Then for any $\sigma \in \cG_\fm$ we have
$$u_\fq(\sigma \xi_{\fm,\fr})=-\sum_{\fr' \mid \fm} e_{\fr}^{\fr'}\cdot u_\fq(\sigma \xi_{\fm/\fr',\fr'}) \text{ in }\cO\otimes_\ZZ \FF_{E(\fn)_\fQ}^\times.$$
(Note that $u_\fq(\sigma \xi_{\fm/\fr',\fr'})$ is annihilated by $M$ and so multiplication by $e_\fr^{\fr'} \in \ZZ/M$ is well-defined.)
\end{lemma}

\begin{proof}
We may assume $\sigma=1$. Let $I_\fn$ be the augmentation ideal of $\ZZ[\cH_\fn]$. For any $g \in (I_\fn + M \ZZ[\cH_\fn])$, we have
$$g D_\fm \omega_\fm \in M \cdot U_{F(\fm)}$$
and we can define
$$\Xi_\fm(g):=\frac 1M g D_\fm \omega_\fm \in U_{F(\fm)}.$$
For $\fr \mid \fm$, define $e(g)^\fr \in \ZZ/M$ by
$$g=e(g)^\fr (\sigma_\fr-1) \text{ in }(I_\fr+M\ZZ[G_\fr])/(I_\fr^2+M\ZZ[G_\fr])\simeq I_\fr/I_\fr^2.$$
Note that by definition
$$\Xi_\fm(1-{\N}\fq{\rm Fr}_\fq^{-1})=\xi_{\fm,\fq} \text{ and }e(1-{\N}\fq {\rm Fr}_\fq^{-1})^\fr= e_\fq^\fr.$$
So the lemma is reduced to the following two claims.
\begin{itemize}
\item[(i)] If $g \in (I_\fn^2+M\ZZ[\cH_\fn])$, then $u_\fq(\Xi_\fm(g))$ vanishes.
\item[(ii)] For any $g \in (I_\fn+M\ZZ[\cH_\fn])$, we have
$$u_\fq(\Xi_\fm(g))=-\sum_{\fr'\mid \fm}e(g)^{\fr'} \cdot u_\fq(\Xi_{\fm/\fr'}(1-{\N}\fr' {\rm Fr}_{\fr'}^{-1})).$$
\end{itemize}

We first show that (i) implies (ii). The right hand side of the equality in (ii), viewed as a function of $g$, factors through $(I_\fn+M\ZZ[\cH_\fn])/(I_\fn^2+M\ZZ[\cH_\fn])$. By (i), so does the left hand side. So it is sufficient to show (ii) for $g=\sigma_\fr-1$ for any prime $\fr \mid \fn$. If $\fr \mid \fm$, we compute
\begin{eqnarray*}
u_\fq(\Xi_\fm(\sigma_\fr-1))&=&u_\fq\left( \frac 1M (\sigma_\fr-1) D_\fm \omega_\fm \right) \\
&\stackrel{(\ref{tele})}{=}&u_\fq\left( \frac 1M (M - {\N}_{G_\fr})D_{\fm/\fr}\omega_\fm \right) \\
&\stackrel{(\ref{omega rel})}{=}& u_\fq(D_{\fm/\fr} \omega_\fm) - u_\fq\left( \frac 1M (1-{\N}\fr {\rm Fr}_\fr^{-1})D_{\fm/\fr}\omega_{\fm/\fr}\right) \\
&\stackrel{(\ref{cong 2})}{=}& -u_\fq(\Xi_{\fm/\fr}(1-{\N}\fr{\rm Fr}_\fr^{-1})).
\end{eqnarray*}
This shows (ii) in this case. If $\fr \nmid\fm$, then $\Xi_\fm(\sigma_\fr-1)=0$ and the equality in (ii) is trivial. Thus we have proved that (i) implies (ii).

We show (i). By (\ref{cong 2}), we see that $u_\fq(\Xi_\fm(M))=0$. Note that $I_\fn$ is generated by $\{\sigma_\fr-1 \mid \fr \mid \fn\}$ over $\ZZ[\cH_\fn]$. By repeating the above computation, we see that
$$u_\fq(\Xi_\fm(h)) \in \langle u_\fq(\tau \Xi_{\fd}(1-{\N}\fq{\rm Fr}_\fq^{-1})) \mid \tau \in \cH_\fn, \  \fd \mid \fm/\fq \rangle_{\ZZ}$$
for any $h \in I_\fn$. Hence, if $g=h h'$ with $h,h' \in I_\fn$, we see that
$$u_\fq(\Xi_\fm(g)) \in \langle u_\fq(\tau \Xi_\fd((1-{\N}\fq{\rm Fr}_\fq^{-1})h)) \mid \tau \in \cH_\fn, \  \fd \mid \fm/\fq \rangle_{\ZZ}.$$
Since $\Xi_\fd((1-{\N}\fq{\rm Fr}_\fq^{-1})h)=(1-{\N}\fq{\rm Fr}_\fq^{-1})\Xi_\fd(h)$ and $1-{\N}\fq{\rm Fr}_\fq^{-1}$ annihilates $\FF_{E(\fn)_\fQ}^\times$, we see that $u_\fq(\Xi_\fm(g))=0$. This proves the claim.
\end{proof}

\begin{proof}[Proof of Lemma \ref{compute}]
(Compare the proof of \cite[Prop. A.13]{MRkoly}.)
It is sufficient to prove
$$ l_{\fr,f}^1(\sigma \kappa'_\fn) =\sum_{\pi\in \mathfrak{S}_1(\fn), \ \pi(\fr)\neq \fr}(-1)^{\nu(\fn/\fd_\pi)}\left( \prod_{\fq \mid \fn/\fd_\pi} e_{\pi(\fq)}^\fq \right) l_{\fr,f}^1(\sigma \kappa'_{\fd_\pi})$$
for any $\sigma \in \Gamma$. We fix a lift of $\sigma$ in $\cG_\fn$, and denote it also by $\sigma$. By Lemma \ref{lem xi}(ii), it is sufficient to show
\begin{eqnarray}\label{formula}
 u_\fr(\sigma \xi_{\fn,\fr}) =\sum_{\pi\in \mathfrak{S}_1(\fn), \ \pi(\fr)\neq \fr}(-1)^{\nu(\fn/\fd_\pi)}\left( \prod_{\fq \mid \fn/\fd_\pi} e_{\pi(\fq)}^\fq \right) u_\fr(\sigma \xi_{\fd_\pi,\fr}).
 \end{eqnarray}
This formula is deduced by using Lemma \ref{lem xi 2} repeatedly. In fact, we apply this lemma to describe $u_\fr(\sigma\xi_{\fn,\fr})$ as a sum of the term $-e_{\fr}^{\fr_1}\cdot u_\fr(\sigma \xi_{\fn/\fr_1,\fr_1})$. We use the lemma again to describe $u_\fr(\sigma \xi_{\fn/\fr_1,\fr_1})$ as a sum of the term $-e_{\fr_1}^{\fr_2}\cdot u_\fr(\sigma \xi_{\fn/\fr_1\fr_2,\fr_2})$. If $\fr_2\neq \fr$, then we repeat the same process. We repeat this until we have $\fr_k=\fr$. We obtain the cyclic permutation $\pi:=(\fr \  \fr_{k-1} \ \cdots \ \fr_2 \ \fr_1) \in \mathfrak{S}_1(\fn)$, and the resulting term is
$$(-1)^{\nu(\fn/\fd_\pi)}\left( \prod_{\fq \mid \fn/\fd_\pi} e_{\pi(\fq)}^\fq \right) u_\fr(\sigma \xi_{\fd_\pi,\fr}).$$
(Note that $\fd_\pi=\fn/\fr\fr_{k-1}\cdots \fr_1$, $\fn/\fd_\pi=\fr \fr_{k-1}\cdots \fr_1$ and $\nu(\fn/\fd_\pi)=k$.) Thus we see that $u_\fr(\sigma \xi_{\fn,\fr})$ is the sum of these terms, and (\ref{formula}) is proved.
\end{proof}

\end{document}